\documentclass[reqno]{amsart}
\newcommand{\RN}[1]{\textup{\uppercase\expandafter{\romannumeral#1}}}
\usepackage{amsmath} 
\usepackage{amssymb}
\usepackage{enumerate}
\usepackage[hidelinks]{hyperref}

\usepackage{mathtools}
\usepackage{thmtools}
\usepackage[capitalise]{cleveref} 
\usepackage{ytableau}
\numberwithin{equation}{section}

\newcommand\mc{\mathcal}
\newcommand\mf{\mathfrak} 
\newcommand\mb{\mathbb}
\crefname{equation}{}{} 
\newtheorem{theorem}{Theorem}[section]  
\newtheorem{lemma}[theorem]{Lemma}
\newtheorem{proposition}[theorem]{Proposition}
\newtheorem{corollary}[theorem]{Corollary}
\newtheorem{definition}[theorem]{Definition}
\newcommand{\Gcd}{\mathrm{gcd}}

\newcommand{\gl}{\mathrm{gl}}

\theoremstyle{definition}
\newtheorem{remark}[theorem]{Remark} 
\usepackage{mathbbol}

\usepackage{combelow}
\usepackage{epigraph}
\usepackage{tikz}
\usepackage{tikz-cd}

\usepackage{enumitem}
\crefname{section}{\S}{\S\S}
\crefname{subsection}{\S}{\S\S}
\title{Instantons on the blown-up surface and the Affine Vertex Algebra}

\begin{document}
\author{Qingyuan Jiang}
\address{Department of Mathematics,
The Hong Kong University of Science and Technology, Clearwater Bay, Kowloon, Hong Kong.} 
\email{jiangqy@ust.hk}

\author{Wei-Ping Li}
\address{Department of Mathematics,
The Hong Kong University of Science and Technology, Clearwater Bay, Kowloon, Hong Kong.} 
\email{mawpli@ust.hk}

\author{Yu Zhao}
\address{School of Mathematics and Statistics, Beijing Institute of Technology, Haidian, Beijing, China}
\email{zy199402@live.com}
\begin{abstract}
We answer a long-standing question raised by Vafa--Witten \cite{Vafa-Witten} on a relation between S-duality and conformal field theory, which related Yoshioka's blow-up formula \cite{Yoshioka1995,Yoshioka1996} and the WZW model for $\mathrm{SU}(r)$ at level $1$. Precisely, for the moduli space of  Euler characteristics of rank $r$ instantons on the blow-up of an algebraic surface along a closed point, we construct the affine $\mathrm{gl}_r$-action on various cohomology theories, including the Grothendieck group of coherent sheaves, Hochschild homology groups, Chow groups and Hodge cohomology groups, and identifying the module as a basic representation.

A key ingredient in our proof is a representation-theoretic reformulation of the theory of Grassmannians of Tor-amplitude $[0,1]$-perfect complexes studied by the first-named author in terms of the spin representation of the finite-dimensional Clifford algebra. This may be viewed as a finite analog of the question of Vafa--Witten via the Boson--Fermion correspondence.
\end{abstract} 
\maketitle
 
\section{Introduction}
 Nakajima \cite{10.1215/S0012-7094-94-07613-8} invented Nakajima quiver varieties and initiated the study of the relation between the moduli space of instantons and vertex operator algebras (VOA for short). Vafa--Witten \cite{Vafa-Witten} tested S-duality (Montonen--Olive duality) via mathematical works like \cite{Mukai,Nakashima,Qin,Yoshioka1995,Yoshioka1996,Klyachko,10.1215/S0012-7094-94-07613-8}, and
found unexpected links between supersymmetric Yang-Mills theory and 2-dimensional conformal field theory (2d CFT for short). In this paper, we confirm a major prediction made by \cite{Vafa-Witten} that relates Yoshioka's blow-up formula to the WZW models of $\mathrm{SU}(r)$ at level $1$.

\subsection{S-duality and modularity} To explain the main results of the paper, we first recall the S-duality conjecture from a mathematical perspective, following \cite{Sduality}. The idea of S-duality can be traced back to Maxwell's equations, which exhibit symmetry under swapping electric and magnetic fields. In the non-abelian case, it was generalized by Montonen--Olive \cite{MontonenOlive} in the $N=4$ supersymmetric gauge theory and is called Montonen--Olive duality, or the S-duality conjecture. In \cite{Vafa-Witten}, Vafa--Witten studied the partition functions of the 4-dimensional $N=4$ supersymmetric theory for a compact 4-dimensional manifold $S$ with gauge group $\mathrm{G}$. When $\mathrm{G}=\mathrm{SU}(r)$ and $S$ is a smooth projective surface with some vanishing conditions, such as
\begin{align}
  \label{assumptionS} 
  & K_S\cong \mc{O}_S \text{ or } K_S\cdot H<0
  \end{align}
which are always satisfied when $S$ is a $K3$ surface, an abelian surface, or a del Pezzo surface, the partition function is the generating function of the Euler characteristics of the moduli space of $H$-stable coherent sheaves on $S$ with rank $r$ and fixed first Chern class\footnote{In general, the partition function is called the Vafa--Witten invariant and a rigorous mathematical definition was proposed by Tanaka--Thomas \cite{Tanaka1,Tanaka2}}. Let $\mathrm{Z}(S,\mathrm{G},\tau)$ be the partition function, with $q=e^{2\pi i \tau}$, where $q$ is the formal variable in the generating function. The S-duality conjecture asserts that 
\begin{equation*}
  \mathrm{Z}(S,\mathrm{G}, \tau)=\pm r^{\chi(S)/2}(\frac{\tau}{i})^{w/2}\mathrm{Z}(S, ^L\mathrm{G}, -1/\tau),
\end{equation*}
where $w$ is a constant that depends only on $S$, $\chi(S)$ is the topological Euler number of $S$, and $^L\mathrm{G}$ is the Langlands dual group of $\mathrm{G}$. When $\mathrm{G}=\mathrm{SU}(r)$, we have $^L\mathrm{G}=\mathrm{SU}(r)/\mu_r$, where $\mu_r$ is the group of the $r$-th root of unity. Then  S-duality conjecture reveals the modularity of the partition function, which suggests an unexpected relation to 2d CFT. 

In particular, Vafa--Witten made several tests based on the mathematical work \cite{Mukai,Nakashima,Qin,Yoshioka1995,Yoshioka1996,Klyachko,10.1215/S0012-7094-94-07613-8}. When $S$ is a K3 surface, the formal series generated by the Euler characteristic is the partition function of the bosonic string by Mukai \cite{Mukai} and Göttsche \cite{gottsche1990betti}, and Vafa--Witten suggested a relation with the Fock space. When $S=\mb{CP}^2$, the partition function is modular only after a natural non-holomorphic modification due to Yoshioka \cite{Yoshioka1995,Yoshioka1996}, and Vafa--Witten suggested that such a modification should be related to a holomorphic anomaly in certain two-dimensional models. When $S$ is an ALE space, Nakajima \cite{10.1215/S0012-7094-94-07613-8} constructed the action of the Kac-Moody algebra on the equivariant cohomology groups of the moduli space of framed sheaves on $S$, through the novel construction of the Hecke correspondences on the Nakajima quiver varieties. Vafa--Witten raised the question to explain Nakajima's work in the framework of S-duality, as ALE spaces are not compact.

In the past decades, there have been substantial advances on the above questions, except the following question which relates the blow-up formula of Yoshioka with the WZW model of $\mathrm{SU}(r)$ at level $1$: when $S$ is a smooth projective surface, let $\hat{S}$ be the blow-up of $S$ along a closed point $o$. Given a total Chern class $c$, let $\mc{M}_H(c)$ and $\mc{M}_{H_{\infty}}(\hat{c})$ be the moduli spaces of Gieseker-stable sheaves on $S$ and $\hat{S}$, respectively (the detailed definition will be given in \cref{sec:1.2}).
   In \cite{Yoshioka1995,Yoshioka1996}, Yoshioka proved a blow-up formula for the Betti numbers of $\mc{M}_{H}(c)$ and $\mc{M}_{H_{\infty}}(\hat{c})$ when the rank is $2$, and if we only consider the Euler characteristics, the formula can be written as
  \begin{equation}
    \label{eq:Yoshioka}
    \frac{\sum_{n}\chi(\mathcal{M}_{H_\infty}(2,p^*c_1-a[C],n))q^{n+a/4}}{\sum_{n}\chi(\mathcal{M}_{H}(2,c_1,n))q^{n}}=\frac{\theta_a(q)}{\eta(q)^2}, \quad a=0,1,
  \end{equation}
   where $p:\hat{S}\to S$ is the projection map, $[C]$ is the exceptional divisor, and 
   \begin{equation*}
    \theta_a(q):=\sum_{n\in \mathbb{Z}+a/2}q^{n^2}, \text{ and } \eta(q):=q^{1/24}\prod_{n\geq 1}(1-q^n)
   \end{equation*}
  are the Jacobi theta function and the Dedekind eta function, respectively. Vafa--Witten \cite{Vafa-Witten} observed that the formula is related to the character of the WZW model of $\mathrm{SU}(2)$ at level $1$, and highlighted the unresolved mystery of the appearance of a 2d CFT in this context and its generalization to higher-rank cases.

  In this paper, we give a full answer to this question by constructing the affine $\mathrm{gl}_r$-action on various cohomology theories of $\mc{M}_{H_\infty}(\hat{c})$ and identifying the module as a basic representation. The detailed statement of the main results will be given in \cref{sec:1.2}.

\subsection{The main results of the paper}
\label{sec:1.2}
  Given an ample line bundle $H$ and a total Chern class $c=(r,c_1,n)$, let $\mc{M}_H(c)$ be the moduli space of Gieseker-stable sheaves on $S$ with total Chern class $c$. For any $1\gg \epsilon>0$, the $\mb{Q}$-divisor $p^*H-\epsilon C$ is also ample, and the moduli space $\mc{M}_{p^*H-\epsilon C}(\hat{c})$ stabilizes as $\epsilon\to 0$ for any given total Chern class $\hat{c}$ on $\hat{S}$; we denote it by $\mc{M}_{H_\infty}(\hat{c})$. Yoshioka's blow-up formula \cref{eq:Yoshioka} was generalized by the second-named author and Qin \cite{Li1999,li1998blowup} to virtual Hodge numbers, where the moduli space may be singular, and by Nakajima--Yoshioka \cite{nakajima2003lectures,1050282810787470592} to arbitrary ranks. For example, for the rank $2$ case we have
		\begin{equation*}
			\frac{\sum_{n}h^{i,j}(\mathcal{M}_{H_\infty}(2,p^*c_1-a[C],n))x^iy^jq^{n+a/4}}{\sum_{n}h^{i,j}(\mathcal{M}_{H}(2,c_1,n))x^iy^jq^{n}} = q^{1/12}Z_a(x,y,q), 
		\end{equation*}
		where $a=0,1$, $h^{i,j}$ is the (virtual) Hodge number and
		\begin{equation}
      \label{theta-fun}
			Z_a(x,y,q)=\frac{\sum_{n\in \mathbb{Z}}(xy)^{\frac{(2n+a)^2-(2n+a)}{2}}q^{(2n+a)^2/4}}{q^{1/12}\prod_{n\geq 1}(1-(xy)^{2n}q^n)^2},
		\end{equation} 
    so that $Z_a(x,1/x,q)=\frac{\theta_a(q)}{\eta(q)^2}$. Nakajima--Yoshioka \cite{nakajima2001moduli} studied the  equivariant cohomology of the moduli space of framed sheaves over $\mb{P}^2$, and constructed a basic representation of the affine Lie algebra $\hat{\mathrm{gl}}_r$ on the blow-up through localization. It provided a rational CFT construction for $S=\mathbb{P}^2$, as there is a correspondence between the WZW model and the basic representation through the affine vertex algebra (we refer to \cite{YellowBook} for details of this correspondence). For the general case, including unframed moduli spaces of stable coherent sheaves on arbitrary surfaces and non-equivariant cohomology, the problem remained open.

In our paper, we give a full answer to this question by providing a geometric construction of this affine Lie algebra $\hat{\mathrm{gl}}_r$ action for any surface $S$ and many different cohomology theories. All the algebras in our paper are defined over $\mb{Z}$. The basic representation of $\hat{\mathrm{gl}}_r$, which is denoted by $\mf{F}(r)$, is defined through the Boson--Fermion correspondence (see \cref{thm:fermion-boson} for the details) and decomposes as the direct sum of irreducible representations $\mf{F}(r)_{l,\bullet}$, where $l$ is the central charge of the representation. It is periodic, which means that $\mf{F}(r)_{l,\bullet}\cong \mf{F}(r)_{l+r,\bullet}$ for any $l\in\mb{Z}$. Let $\mb{H}^*$ be one of the (co)homology theories, including the Grothendieck group of coherent sheaves, the Hochschild homology group, the Chow group, and the Hodge cohomology group.
\begin{theorem}[\cref{main1}] \label{thm:1.1} Fix $r$ and $c_1$ such that $\Gcd(r,c_1\cdot H)=1$. Then for any $l\in\mb{Z}$, we have an identification of $\hat{\mathrm{gl}}_r$-representations:
  \begin{equation}
    \label{eq:1.1111}
    \bigoplus_{n\in \mb{Z}}\mb{H}^*(\mc{M}_{H_\infty}(r,p^*c_1+l[C],n-\frac{l(l+1)}{2}))\cong \bigoplus_{n\in \mb{Z}}\mb{H}^*(\mc{M}_{H}(r,c_1,n))\otimes \mf{F}(r)_{l,\bullet}.
  \end{equation} 
  Here,  the assumption \cref{assumptionS} will be needed for the Hochschild homology group, the Chow group, and the Hodge cohomology group.
\end{theorem}
Our $\hat{\mathrm{gl}}_r$ acts trivially on the factor $\bigoplus_{n\in \mb{Z}}\mb{H}^*(\mc{M}_{H}(r,c_1,n))$ on the right-hand side, which is one of the main differences from the previous work \cite{nakajima1997heisenberg,neguct2018hecke}. For the Hochschild homology group, the action preserves the homological degrees, which explains why $Z_a(x,1/x,q)$ in \cref{theta-fun} is independent of $x$.

\subsection{Main difficulties and outline of the proof of \cref{thm:1.1}}
\label{sec:1.3}
There are several essential difficulties in constructing the actions like  \cref{eq:1.1111}. A direct Hecke-correspondence construction is tempting but fails for a basic reason: the Hecke modification along the exceptional curve,
\begin{equation}
  \label{eq:Hecke}
    0\to \mc{E}_0\to \mc{E}_1\to \mc{O}_C\to 0,
\end{equation} 
changes the first Chern class, whereas the $\hat{\mathrm{gl}}_r$-action preserves it. Moreover, such a modification may change the stability of $\mc{E}_0$ and $\mc{E}_1$, making the correspondences not well-defined. Thus, unlike the constructions in \cite{nakajima1997heisenberg,neguct2018hecke}, the correspondence in \cref{eq:Hecke} does not define an operator on a fixed moduli space of stable sheaves on $\hat{S}$. 
These difficulties were already visible in the early approaches to this problem \cite{Li-Qin,Li-Qin-2,nakajima2001moduli}. Our proof resolves them through the following steps.

\medskip
\noindent\textbf{Step 1. A finite Clifford--Grassmannian model.}
The first ingredient is a finite Clifford--Grassmannian model, which is of independent interest. Recall that a perfect complex $\mc{E}$ of Tor-amplitude $[0,1]$ on a scheme $X$ is one that is locally represented by a morphism of finite locally free sheaves,
\[
  \mc{E}\cong \{V\xrightarrow{s} W\},
\]
and its rank is $\mathrm{rank}(W)-\mathrm{rank}(V)$. The shifted dual $\mc{E}^{\vee}[1]$ is again of Tor-amplitude $[0,1]$, locally represented by the dual morphism $s^\vee\colon W^\vee \to V^\vee$. 

Following Grothendieck's convention, we consider the moduli space $\mathrm{Gr}_X(\mc{E},d)$ parametrizing quotients $h^0(\mc{E})\twoheadrightarrow Q$, where $Q$ is a locally free sheaf of rank $d$.

For a perfect complex $\mc{E}$ of Tor-amplitude $[0,1]$ and rank $r>0$, we prove an isomorphism of graded $\mathrm{Cl}(r)$-representations,
\begin{equation}
\label{eq:twice}
    \bigoplus_{d \in \mathbb{Z}} \mb{H}^*(\mathrm{Gr}_X(\mathcal{E}, d))
    \cong
    \bigoplus_{d \in \mathbb{Z}} \mb{H}^*(\mathrm{Gr}_X(\mathcal{E}^{\vee}[1], d)) \otimes \mathrm{F}(r).
\end{equation}
Here $\mathrm{F}(r)=\bigoplus_{i=0}^{r}\wedge^i\mathbb{Z}^r$, and $\mathrm{Cl}(r)$ (defined in \cref{sec:cli}) is the finite Clifford algebra generated by contraction and wedge operators
\[
  p_i=\frac{\partial}{\partial v_i},\qquad q_i=v_i\wedge -,\qquad 0\leq i\leq r-1.
\]
We refer to \cref{thm:1}, \cref{thm:4.4}, and \cref{thm:hos} for the detailed formulations. The operators $g_{ij}=p_iq_j$ give the standard $\mathrm{gl}_r$-action on $\mathrm{F}(r)$. Thus \cref{eq:twice} provides a finite, Grassmannian analogue of the Boson--Fermion correspondence. 

This construction is new even when $\mc{E}$ is a vector bundle of rank $r$, in which case it gives an explicit identification, through Clifford algebra representations, of the cohomology of Grassmannian bundles with the Grassmann algebra:
\[
    \bigoplus_{d \in \mathbb{Z}} \mb{H}^*(\mathrm{Gr}_X(\mathcal{E}, d))
    \cong
    \mb{H}^*(X)\otimes \bigoplus_{i=0}^{r}\wedge^i\mathbb{Z}^r.
\]

The proof of \cref{eq:twice} uses semiorthogonal decompositions in a representation-theoretic way: a semiorthogonal decomposition is interpreted as a categorical representation of an associative algebra. 
We refer to \cref{lem:515,prop:4} for more details.

\medskip
\noindent\textbf{Step 2. Perverse coherent sheaves and the finite type algebra $\mc{E}(r)$.} 
Following Bridgeland \cite{Bridgeland2007} and Nakajima--Yoshioka \cite{1050282810787470592}, we consider the moduli space $\mathrm{M}^m(l,n)$ of $m$-stable perverse coherent sheaves on $\hat S$ with total Chern class
\begin{equation*}
      \left(r,p^*c_1+l[C],n-\frac{l(l+1)}{2} \right)\in \mb{Z}_{>0}\times \mathrm{NS}(\hat{S})\times \mb{Z}.
\end{equation*}
For fixed $l,n$, $\mathrm{M}^m(l,n)$ and $\mathrm{M}^{m+1}(l,n)$ are related by wall-crossing, while for $m\gg 0$, one recovers  the moduli space of stable sheaves on the blow-up:
\[
  \mathrm{M}^{m}(l,n)\cong \mc{M}_{H_\infty}(l,n): = 
  \mc{M}_{H_\infty}\bigl(r,p^*c_1+l[C],n-\frac{l(l+1)}{2}\bigr),
\]
where the right-hand side denotes the moduli space of $(p^*H-\epsilon C)$-stable sheaves on $\hat{S}$ for $0 < \epsilon \ll 1$.

Wall-crossing alone, however, does not suffice to reveal the representation-theoretic structure of these moduli spaces, and an essentially new approach is required: we first construct
the correspondences on $\mathrm{M}^0(l,n)$ for all $l,n\in \mb{Z}$, and then propagate the resulting representations to the $\mc{M}_{H_\infty}(l,n)$ using the  limit property and the self-similarity 
$ \mathrm{M}^{m}(l,n)\cong \mathrm{M}^{m-1}(l-r,n-l)$
induced by tensoring with $\mc{O}_{\hat S}(-C)$.

To construct the correspondences on $\mathrm{M}^0(l,n)$, we first observe that these moduli spaces admit an explicit Grassmannian description. By \cite{jiang2022grassmanian} and \cite{1050282810787470592}, one has
\begin{equation}
  \label{eq:grassmannian}
  \mathrm{M}^0(l,n)\cong\mathrm{Gr}_{\mc{M}_{H}(r,c_1,n)}(\mc{U}_{o},-l), 
  \quad
  \mathrm{M}^0(l,n)\cong \mathrm{Gr}_{\mc{M}_{H}(r,c_1,n-l)}(\mc{U}_{o}^{\vee}[1],-l),
\end{equation}
where $\mc{U}_o$ is the restriction of the universal sheaf to
$\mc{M}_{H}(r,c_1,n)\times\{o\}$. 
These are precisely the two presentations appearing in the finite Clifford--Grassmannian model of Step 1, so applying \cref{eq:twice} yields an isomorphism 
\begin{equation*}
  \bigoplus_{l\in\mb{Z}} \bigoplus_{n\in \mb{Z}}\mb{H}^*(\mathrm{M}^0(l,n))\cong \bigoplus_{l\in\mb{Z}} \bigoplus_{n\in \mb{Z}}\mb{H}^*(\mathrm{M}^0(l,n)) \otimes \mathrm{F}(r).
\end{equation*}
Moreover, the embedding of the degree-$0$ vector $1\in \mathrm{F}(r)$ produces two additional operators $e$ and $f$ on $\bigoplus_{l,n\in\mb{Z}}\mb{H}^*(\mathrm{M}^0(l,n))$. Together with the Clifford generators $p_i,q_i$ of $\mathrm{Cl}(r)$, these operators generate a larger algebra $\mc{E}(r)$, with relations $p_ie=0$ and $fq_i=0$ for all $0\leq i\leq r-1$, and $fe=1$ and $ef=\prod_{j=0}^{r-1}p_jq_j$.  


\medskip
\noindent\textbf{Step 3. Morita equivalence, stable colimits, and the infinite Clifford algebra $\mf{E}(r)$.}
The passage from $\mc{E}(r)$-representations to representations of the infinite Clifford algebra is achieved through a Morita equivalence.

Let $\mf{E}(r)$ (defined in \cref{sec:666}) be the infinite Clifford algebra generated by $P_{a,i}, Q_{a,i}$, where $a\in \mb{Z}, 0 \le i \le r-1$, with relations 
\begin{equation*}
      \{P_{a,i}, P_{b,j}\}=0,\qquad
      \{Q_{a,i}, Q_{b,j}\}=0,\qquad
      \{P_{a,i}, Q_{b,j}\}=\delta_{ij}\delta_{ab},
\end{equation*}
and an extra invertible shift operator $E$ satisfying
\[
  EP_{a,i}=P_{a+1,i}E,\qquad EQ_{a,i}=Q_{a+1,i}E.
\]
The usual Fermionic Fock representations $V$ are characterized by a two-sided admissibility condition: for any $v\in V$, we have $P_{a,i}v=0$ and $Q_{-a,i}v=0$ for all $a\gg 0$. 
For the geometric applications considered here, however, one needs a weaker
one-sided (right-admissible) condition: for every vector $v$, one only requires
\[
  P_{a,i}v=0 \qquad \text{for all } a\gg 0.
\]
Perhaps unexpectedly, our algebraic result \cref{cor:equivalence1} establishes that the category of such representations of the infinite Clifford algebra $\mf{E}(r)$ is Morita equivalent to the category of representations of the finitely generated algebra $\mc{E}(r)$.

In \cref{prop:limit},  we prove that under this Morita equivalence, the image of the $\mc{E}(r)$-module $\bigoplus_{l,n}\mb{H}^*(\mathrm{M}^0(l,n))$ admits two equivalent descriptions as an $\mf{E}(r)$-module: it is isomorphic both to $\bigoplus_{l,n}\mb{H}^*(\mc{M}_{H_\infty}(l,n))$, and $\bigoplus_{n} \mb{H}^*(\mc{M}_{H}(r,c_1,n))\otimes \mf{F}(r)$, respectively, where $\mf{F}(r)$ is the Fermionic Fock space.


\medskip
\noindent\textbf{Step 4. The Boson--Fermion correspondence and the $\hat{\mathrm{gl}}_r$-representation.}
The Boson--Fermion correspondence of I. Frenkel \cite{FRENKEL1981259} identifies the Fermionic Fock space $\mf{F}(r)$ with the basic representation of the affine Lie algebra $\hat{\mathrm{gl}}_r$.  Concretely, the affine generators are obtained from the normal-ordered quadratic fermions
\[
:Q_{m,i}P_{n,j}:
\]
Transporting, via this correspondence, the $\mf{E}(r)$-representations of the previous step yields the desired $\hat{\mathrm{gl}}_r$-action on the cohomology of the blow-up moduli spaces. Moreover, restricting to the central charge-$l$ summand gives precisely
\[
  \bigoplus_{n\in \mb{Z}} \mb{H}^*(\mc{M}_{H_\infty}(l,n))
  \cong
  \left(
  \bigoplus_{n\in \mb{Z}}\mb{H}^*(\mc{M}_{H}(r,c_1,n))
  \right)\otimes \mf{F}(r)_{l,\bullet},
\]
which is \cref{eq:1.1111}. 
This completes the conceptual route of our proof of \cref{thm:1.1}.

\bigskip
In this paper, the interplays between algebra and geometry are paramount. Special geometric features of the moduli spaces dictate what kind of algebra is involved. Consequently, three types of algebras appear
sequentially, with one being an extension of the previous one. Representation-theoretic results of these algebras, in turn, give the anticipated topological properties of moduli spaces.

\subsection{The comparison with the previous work}

\label{sec:1.7}
Since \cite{10.1215/S0012-7094-94-07613-8,Vafa-Witten}, substantial advances have been made by both mathematicians and physicists towards understanding the relationship between the moduli space of instantons and the VOA. A pivotal first step was taken by Nakajima \cite{nakajima1997heisenberg} and Grojnowski \cite{grojnowski1996instantons}, who proved that the cohomology groups of the Hilbert schemes of points on an algebraic surface are the Fock space of a (super)-Heisenberg algebra.  Lehn \cite{Lehn} constructed the Virasoro algebra, and the second-named author, Qin and Wang \cite{Li-Qin-Wang} constructed the $W_{1+\infty}$-algebra action on the cohomology groups of the Hilbert schemes of points on an algebraic surface.

The breakthrough on quiver varieties was also obtained by Nakajima \cite{10.1215/S0012-7094-98-09120-7,10.2307/2646205}, where he constructed the affine Kac-Moody algebra and quantum loop algebra actions on the equivariant cohomology groups and Grothendieck groups of coherent sheaves over the Nakajima quiver varieties. Varagnolo constructed the Yangian action on the equivariant cohomology group of quiver varieties \cite{Varagnolo1999}.

The Hilbert scheme of points is also the moduli space of rank $1$ stable coherent sheaves on a surface. For the higher rank generalization, one of the most important advances is the Alday--Gaiotto--Tachikawa correspondence \cite{Alday2010}, which asserts that the equivariant cohomology of the moduli space of framed sheaves on $\mathbb{P}^2$ is the Verma module of the $W$-algebra. It was established by Schiffmann--Vasserot \cite{MR3150250}, Maulik--Okounkov \cite{MR3951025}, and Braverman--Finkelberg--Nakajima \cite{braverman2016instanton}. A $\mathrm{K}$-theoretic generalization for any algebraic surface was developed by Negu\cb{t} \cite{neguct2018hecke}.

When the assumption \cref{assumptionS} is not satisfied, the partition function will not be the same as the generating function of the Euler characteristics of the moduli space of stable coherent sheaves, as a new contribution from the ''monopole branch'' will appear. When $S$ is a projective algebraic surface, such a definition is given by Tanaka and Thomas \cite{Tanaka1,Tanaka2} using the moduli space of stable Higgs pairs. The representation theoretic explanation for the monopole branch is an interesting question, and we will pursue it in our future work.

 In addition to the above novel breakthroughs, several works also directly inspired our approach. The first work is by Koseki \cite{koseki2021categorical}, where he constructed an SOD of the derived category $\mathrm{D}^b_{coh}(\mc{M}_{H_\infty}(\hat{c}))$ using a theorem of Toda \cite{MR4608555}, under a smoothness assumption. When the third-named author was working at IPMU, Nakajima suggested that the third-named author consider its relation to representation theory, which became the starting point of this paper. The results in \cite{koseki2021categorical} are not immediately applicable to representation theory, as the explicit Fourier--Mukai kernels by the second-named author in \cite{jiang2023derived} are essential for the proof of \cref{thm:1.1}, which not provided in \cite{koseki2021categorical,MR4608555}.

 The second work is by Bershtein--Feigin--Litvinov \cite{Bershtein2016}, where they explained (part of) the rank $2$ Nakajima--Yoshioka blow-up equations \cite{NYblowup} on $S=\mb{P}^2$ as an isomorphism of conformal blocks, using a different approach from ours. Their vertex algebra is the same as ours, but the conformal vector is different. To compare the different conformal vectors, one has to study the variation of Donaldson invariants under the geometric correspondences, which will be pursued in our future work.

 The third work is by Diaconescu--Porta--Sala--Schiffmann--Vasserot \cite{Diaconescu2025,Diaconescu2026}, where they started the study of the cohomological Hall algebra of the moduli stack of rank $1$ coherent sheaves on a surface, and moreover related it to Yangians. While our work also involves the Hecke modification along the $(-1)$-curve $C$, our approach is different from theirs: the operators $e,f$ in our paper are defined through the two Grassmannian presentations of $\mathrm{M}^0(l,n)$, which cannot be directly obtained from the Hecke correspondences. Moreover, the limit property of $\mathrm{M}^m(l,n)$ is an essential ingredient for the proof of \cref{thm:1.1}, which also cannot be obtained solely from the Hecke correspondences.

 Finally, we would like to mention the work by the third-named author \cite{zhao2024hilbert}, where he studied a toy model of the Hilbert scheme of points on the blown-up surface. Compared to \cite{zhao2024hilbert}, \cref{thm:1.1} has physical applications due to its connections with the WZW model and  S-duality.  From a technical point of view, the proof of \cref{thm:1.1} is also significantly more involved than that in \cite{zhao2024hilbert} for two reasons: the Clifford algebras that act in the rank $r$ case are more complicated than the ones that act in the rank $1$ case, and the geometry of the correspondences is more complicated than that of \cite{zhao2024hilbert}, which we have discussed in \cref{sec:1.3}.
\subsection{The organization of the paper and some notation} In \cref{sec:gra}, we recall the moduli space of stable perverse coherent sheaves of Nakajima--Yoshioka \cite{1050282810787470592}. The most technical part of our paper is proving \cref{eq:twice} and its cohomological variants, which take up three sections from \cref{sec:cli} to \cref{sec:derived}. From \cref{sec:666} to \cref{sec:fermion}, we construct the $\mc{E}(r)$-action on the algebraic $\mathrm{K}$-theory/cohomology of the moduli space of stable perverse coherent sheaves on the blow-up. We also define the functor $\mf{H}_{\infty}$ and prove that $\mf{H}_{\infty}(\bigoplus_{l,n}\mb{H}^*(\mathrm{M}^0(l,n)))$ has two realizations in terms of the cohomology groups of the moduli space of coherent sheaves on blow-ups and the Fermionic Fock space, respectively. In \cref{sec:main}, we recall the Boson--Fermion correspondence and prove the main theorem.

 Let $R$ be a unital associative algebra. We denote the opposite ring of $R$ by $R^{op}$ and refer to a left $R$-module as an $R$-representation. Moreover, given $x,y\in R$, we write $[x,y]:=xy-yx$ and $\{x,y\}:=xy+yx$. Given a finite family of elements $\{a_i\}_{i\in S}$ in a unital ring $R$, if $S$ has a total order $P$, we sort all $a_i$ as $a_{i_1},a_{i_2},\ldots,a_{i_{|S|}}$ such that $i_1<_P i_2<_P \cdots <_P i_{|S|}$, and
  define the ordered product
  \begin{equation*}
    \prod_{i\in S}^{P}a_i:=\prod_{j=1}^{|S|}a_{i_j}.
  \end{equation*}

  Given any two elements $a,b$ in a set $S$, we use the Kronecker delta symbol $\delta_{a,b}=1$ if $a=b$ and $\delta_{a,b}=0$ if $a\neq b$. 

  As in \cref{thm:1.1}, we use $\mb{H}^*$ to denote one of the (co)homology theories, including the Grothendieck group of coherent sheaves, the Hochschild homology group, the Chow group, and the Hodge cohomology group.
  As in \cref{eq:twice}, we denote the moduli space of $d$-dimensional quotient spaces of $\mathbb{C}^r$ by $\mathrm{Gr}(r,d)$.

 \subsection{Acknowledgments}
As discussed in \cref{sec:1.7},
we learned of this question through discussions with Hiraku Nakajima. We would like to thank him for suggesting this problem and for his valuable guidance throughout our work in this field.
   Y.Z. would like to thank Mikhail Bershtein, Bin Gui, Henry Liu, Andrei Negu\c{t},  and Xin Sun for many helpful discussions and the Bernoulli Center, EPFL for the invitation to the conference "Representations, Moduli and Duality". Q. J. would also like to thank Naoki Koseki and Sasha Minets for helpful conversations.
   W.P. Li is supported by GRF grants 16305125 and 16306222. Q. J. was supported by the Hong Kong Research Grants Council ECS grant (Grant No. 26311724).
   
\section{Perverse coherent sheaves on the blow-ups} 
\label{sec:gra}
In this section, we briefly review the moduli space of stable perverse coherent sheaves on the blow-up surface constructed by Nakajima--Yoshioka \cite{1050282810787470592}.

Let $S$ be a smooth projective surface over $\mathbb{C}$, with an ample divisor $H$ and $(r,c_1)\in \mb{Z}_{>0}\times \mathrm{NS}(S)$ such that $\mathrm{gcd}(r,c_1\cdot H)=1$. Let $\mc{M}_{H}(n)$ be the moduli space of $H$-stable coherent sheaves on $S$ with total Chern class $(r,c_1,n)$, which is a projective scheme. The universal sheaf $\mc{U}$ over $\mc{M}_{H}(n) \times S$, regarded as a perfect complex, is a Tor-amplitude $[0,1]$-perfect complex of rank $r$ (we refer to \cite{huybrechts2010geometry} and \cite{neguct2018hecke} for references). We fix a closed point $o\in S$.
Let $\mc{U}_{o}:=\mc{U}|_{\mc{M}_{H}(n)\times \{o\}}$ be the restriction of $\mc{U}$ to the fiber $\mc{M}_{H}(n)\times \{o\}$. 

Let $\hat{S}:=\mathrm{Bl}_{o}S$ be the blow-up of $S$ at $o$ with the exceptional divisor $C$ and the projection map $p:\hat{S}\to S$. Then $\mathrm{NS}(\hat{S})=p^*\mathrm{NS}(S)\oplus \mathbb{Z}[C]$ where $[C]$ is the fundamental class of $C$ in $\hat{S}$. For any rational number $1\gg\epsilon>0$, $p^*H-\epsilon C$ is $\mathbb{Q}$-ample. Given $l,n\in \mb{Z}$, the moduli spaces of $(p^*H-\epsilon C)$-stable coherent sheaves on $\hat{S}$ with the total Chern class
\begin{equation}
  \label{eq:moduli}
  (r,p^*c_1+l[C],n-\frac{l(l+1)}{2})\in \mb{Z}_{>0}\times \mathrm{NS}(\hat{S})\times \mb{Z}
\end{equation}
are all isomorphic when $\epsilon$ is sufficiently small, which are denoted by $\mc{M}_{H_{\infty}}(l,n)$. 
The corresponding Chern character of \cref{eq:moduli} is $(r,p^*c_1+l[C],-n+\frac{l}{2}+\frac{c_1^2}{2})$. 

Now we recall the definition of the stable perverse coherent sheaf. For any $m\in \mathbb{Z}$, let $\mc{O}_C(m):=\mc{O}_{\hat{S}}(-mC)|_{C}$.
A coherent sheaf $E$ on $\hat{S}$ is called a stable perverse coherent sheaf if $p_{*}E$ is $H$-stable on $S$ with the vanishing conditions
    \begin{equation*}
      \operatorname{Hom}(E,\mc{O}_{C}(-1))=0,\quad \operatorname{Hom}(\mc{O}_{C},E)=0,
    \end{equation*}
and is called $m$-stable if $E\otimes \mc{O}(-mC)$ is a stable perverse coherent sheaf. Given $l,n\in \mathbb{Z}$, we define  $\mathrm{M}^m(l,n)$ as the moduli space of $m$-stable perverse coherent sheaves on $\hat{S}$ with the total Chern class \cref{eq:moduli}.
\begin{theorem}[Proposition 3.1 of \cite{kuhn2021blowup}, Lemma 3.4 and Theorem 4.1 of \cite{1050282810787470592}] \label{thm:NY1}
  Given $E$ in $\mathrm{M}^0(l,n)$ and $a=0,1$, the higher direct image $\mathrm{R}^{i}p_{*}E(aC) = 0$ for all $i>0$, and the pushforward $p_{*}E(aC)$ is $H$-stable. Moreover, let 
  \begin{align*}
   \zeta: \mathrm{M}^{0}(l,n) \to \mc{M}_{H}(n-l)
  \quad \eta: \mathrm{M}^{0}(l,n) \to \mc{M}_{H}(n)
  \end{align*}
  be the morphisms which map $E$ to $p_*E$ and $p_*(E(C))$ respectively. Then $\zeta$ and $\eta$ are identified with the natural projections from $\mathrm{Gr}_{\mc{M}_{H}(n-l)}(\mc{U}_{o}^{\vee}[1],-l)$ to $\mc{M}_{H}(n-l)$ and from $\mathrm{Gr}_{\mc{M}_{H}(n)}(\mc{U}_{o},-l)$ to $\mc{M}_{H}(n)$ respectively.
\end{theorem}
By \cref{thm:NY1}, we have $\mathrm{M}^0(0,n)\cong \mc{M}_{H}(n)$ for any $n\in \mb{Z}$.
\begin{theorem}[Proposition 3.37 of \cite{1050282810787470592}]
  \label{thm:NY2}
  Fixing integers $l,n$, for $m\gg 0$ we have 
  \begin{equation*}
    \mathrm{M}^0(l-mr,n-ml+\frac{m(m-1)}{2}r)\cong \mathrm{M}^{m}(l,n)\cong \mc{M}_{H_{\infty}}(l,n).
  \end{equation*}
\end{theorem}
\begin{proof}
  The second isomorphism $\mathrm{M}^{m}(l,n)\cong \mc{M}_{H_{\infty}}(l,n)$ when $m\gg 0$ is precisely Proposition 3.37 of \cite{1050282810787470592}. For the first isomorphism, we only need to notice that given $E\in \mathrm{M}^m(l,n)$, $E\otimes \mc{O}(-mC)$ is $0$-stable with the Chern character 
\begin{align*}
  (r,p^*c_1+(l-mr)[C],-n+ml-\frac{m^2}{2}r+\frac{l}{2}+\frac{c_1^2}{2})
\end{align*}
and hence is in $\mathrm{M}^0(l-mr,n-ml+\frac{m(m-1)}{2}r)$.
\end{proof}

\begin{theorem}[Lemma 3.22 of \cite{1050282810787470592} and Lemma 3.3 of \cite{koseki2021categorical}]
\label{thm:smooth} Assuming \cref{assumptionS},
For any $l,n\in \mb{Z}$, $\mathrm{M}^{0}(l,n)$ is either empty or a smooth variety of dimension 
  \begin{equation*}
    \mathrm{const}+2rn-l(l+r),
  \end{equation*}
  where $\mathrm{const}:=(1-r)c_1^2-(r^2-1)\chi(\mc{O}_S)+h^1(\mc{O}_S)$ only depends on $S$, $r$ and $c_1$.
  \end{theorem}

\begin{remark} We remind the reader that our convention is slightly different from \cite{1050282810787470592}, where they considered the Chern character $(r,p^*c_1-l[C],-n+\frac{l}{2}+\frac{c_1^2}{2})$, as we need to fit with the extended Clifford algebra action in \cref{sec:666}. 
\end{remark}

\section{Grassmannians of complexes and the Clifford algebra}
\label{sec:cli}
For any integer $r>0$, let the \textbf{rank $r$ Clifford algebra $\mathrm{Cl}(r)$} be the unital associative $\mathbb{Z}$-algebra generated by elements $p_i$ and $q_i$ for $i \in [r]$, subject to the relations $p_i^2=q_i^2=0$ for all $i \in [r]$, and 
\begin{equation*}
 \{p_i, p_j\} = 0, \quad \{q_i, q_j\} = 0, \quad
\{p_i, q_j\} = \delta_{ij} \cdot 1  \quad \text{for all } i, j \in [r].
\end{equation*} 
Let $\mathrm{F}(r):=\bigoplus_{d=0}^{r}\wedge^d(\mb{Z}^r)$ be the exterior algebra and $\{v_i \mid i\in [r]\}$ be a standard basis of $\mathbb{Z}^r$.
Then $\mathrm{F}(r)$ is the \textbf{spin representation} of $\mathrm{Cl}(r)$ by the action
  \begin{equation}
    \label{eq:rep1}
    p_i:=\frac{\partial}{\partial v_i}, \quad q_i:=v_i\wedge -.
\end{equation}

In this section, we first sketch a $\mathrm{Cl}(r)$-action on $\bigoplus_{d=0}^{r} \mb{H}^*(\mathrm{Gr}(r,d))$ for some (co)homology theory $\mb{H}^*$. We only need an identification of $\mathbb{Z}$-modules 
\begin{equation}
  \label{eq:model}
    \bigoplus_{d=0}^{r} \mb{H}^*(\mathrm{Gr}(r,d))\cong \mathrm{F}(r).
\end{equation}
We notice that the basis of $\mathrm{F}(r)$ is parametrized by all the subsets of $[r]$, and for any $I\subset [r]$, we define the \textbf{weight} and \textbf{length} of $I$ as $\mathrm{wt}(I):=\sum_{i\in I}i$ and  $\mathrm{len}(I):=\sum_{i\in I}1$. On the other hand, a \textbf{partition} $\lambda$ is a sequence of positive integers $\lambda=(\lambda_1\geq \lambda_2\geq \cdots \geq \lambda_l)$. Given any two integers $a,b\geq 0$, let $B_{a,b}$ be the set of partitions with  $l\leq a$ and $\lambda_1\leq b$. By the Schubert calculus, the basis of $\mb{H}^*(\mathrm{Gr}(r,d))$ can be parametrized by all partitions in $B_{d,r-d}$. Let $\mathrm{P}_r$ be the disjoint union of all $B_{d,r-d}$ for $0\leq d\leq r$. On the other hand, let $[r]:=\{0,\cdots,r-1\}$ be the set of $r$ elements. Then \cref{eq:model} is naturally induced from the bijection $\iota$ from all the subsets of $[r]$ to $\mathrm{P}_r$
  \begin{align}
  \label{eq:iotaa}
   \iota: (i_0<i_1<\cdots <i_{d-1})\to (i_{d-1}-d+1\geq \cdots \geq i_0).
\end{align}

From \cref{sec:cli} to \cref{sec:derived}, we generalize \cref{eq:model} to the Grassmannians of Tor-amplitude $[0,1]$-complexes and also to different (co)homology theories. In particular, we construct an isomorphism of graded $\mathrm{Cl}(r)$-representations: 
\begin{equation}
\label{eq:once}
    \bigoplus_{d \in \mathbb{Z}} \mb{H}^*(\mathrm{Gr}_X(\mathcal{E}, d))
    \cong
    \bigoplus_{d \in \mathbb{Z}} \mb{H}^*(\mathrm{Gr}_X(\mathcal{E}^{\vee}[1], d)) \otimes \mathrm{F}(r),
\end{equation}
where $\mb{H}^*$ can be the Chow groups (\cref{thm:4.4}), Hodge cohomology groups (\cref{thm:4.4}), Grothendieck group of coherent sheaves (\cref{thm:1}) and Hochschild homology groups (\cref{thm:hos}). The main tool for constructing \cref{eq:once} arises from the semi-orthogonal decomposition in derived categories, where it naturally induces a decomposition of the Grothendieck groups which is however generally not orthogonal and needs an \textbf{orthogonalization} procedure. Surprisingly, we find that this orthogonalization is particularly useful for constructing $\mathrm{Cl}(r)$-representations, which we explain in this section. We will give the geometric correspondences in \cref{sec:chow} and \cref{sec:derived}.

\begin{remark}
 Unlike $\mathrm{gl}_r$, where there are infinitely many irreducible representations, $\mathrm{F}(r)$ is the unique irreducible representation of $\mathrm{Cl}(r)$ (see \cref{lem:irreducible} for a proof). We will return to this perspective in \cref{sec:main}, via the Boson--Fermion correspondence.
\end{remark}

\subsection{The orthogonalization and \texorpdfstring{$\mathrm{Cl}(r)$-representations}{}}
\label{sec:444} In $\mathrm{Cl}(r)$ we define $p_{\emptyset} = q_{\emptyset}=1$ and for any subset $I=(a_0< \ldots< a_k) \subset [r]$ we define
\begin{equation*}
p_{I}:=p_{a_{k}}p_{a_{k-1}}\cdots p_{a_0}, \quad q_{I}:=q_{a_0}q_{a_1}\cdots q_{a_k}.
\end{equation*}
\begin{lemma}
\label{lem:irreducible}
Let $V$ be a $\mathrm{Cl}(r)$-representation and let $W:= \bigcap_{i\in [r]}\ker(p_i)$.
 Then we have the isomorphism of $\mathrm{Cl}(r)$-representations
\begin{equation}
  \label{eq:clifcong}
  W\otimes\mathrm{F}(r)\to V, \quad w\otimes v_I\mapsto q_Iw.
\end{equation}
\end{lemma}
\begin{proof}
As $\mathrm{F}(r)\cong \bigotimes_{i=1}^r \mathrm{F}(1)$ as $\mathbb{Z}$-modules, by induction we only need to prove the case that $r=1$. In this case, $\mathrm{Im}(p_0)\subset W$ and the following morphism 
  \begin{equation*}
      V\to W\otimes \mathrm{F}(1) \quad v\to p_0v\otimes v_{[1]}+p_0q_0v\otimes v_{\emptyset}.
  \end{equation*}
  is the inverse of \cref{eq:clifcong} by direct computations.
\end{proof}
By \cref{lem:irreducible}, many properties of a $\mathrm{Cl}(r)$-representation $V$ can be reduced to the computations on $\mathrm{F}(r)$. For example, for any $\mathrm{Cl}(r)$-representation $V$, we have
\begin{equation}
  \label{lem:45}
  (p_{[r]}q_{[r]})|_{(\cap_{i\in [r]}\mathrm{ker}(p_i))}=\mathrm{Id}, \quad (q_{[r]}p_{[r]})|_{(\cap_{i\in [r]}\mathrm{ker}(q_i))}=\mathrm{Id},
\end{equation}
and $\sum_{I\subset [r]}q_Ip_{[r]}q_{[r]}p_I=\mathrm{Id}_V$. Regarding $\mathrm{Cl}(r)$ as a $\mathrm{Cl}(r)$-representation via left multiplication, in $\mathrm{Cl}(r)$ we have
\begin{equation}
      \label{cor:faith}
      \sum_{I\subset [r]}q_{I}p_{[r]}q_{[r]}p_{I}=1.
\end{equation}

Now let $W$ and $V$ be two $\mathbb{Z}$-modules and $\mc{P}$ be a total order on the set of all subsets $I\subset [r]$. Given morphisms $\mf{e}_I:W\to V$ and  $\mf{f}_I:V\to W$ for all $I\subset [r]$, we say that they are \textbf{semi-orthogonal} with respect to $\mc{P}$ if
\begin{equation*}
    \mf{f}_I\mf{e}_J=\delta_{IJ} \cdot \operatorname{id}_{W} \quad \text{for all} \quad I\geq_{\mc{P}}J.
\end{equation*}
We define their \textbf{orthogonalization} as 
\begin{gather}
    \label{e9}
    e_I:=\mf{e}_I, \quad f_{I}:=\mf{f}_I\prod_{J>_{\mc{P}}I}^{\mc{P}}(\mathrm{Id}_{V}-\mf{e}_J\mf{f}_J).
  \end{gather}
We collect all morphisms $e_I$ and $f_I$ and define the morphisms
 \begin{gather}
  \label{eq:ehat}
  \hat{e}:W\otimes\mathrm{F}(r)\to V, \quad x\otimes v_I\mapsto e_Ix, \\
  \nonumber
  \hat{f}:V\to W\otimes\mathrm{F}(r), \quad y\mapsto \sum_{I\subset [r]}f_Iy\otimes v_I.
\end{gather}
The surprising fact is that the morphism $\hat{e}$ (or $\hat{f}$) is an isomorphism of $\mathrm{Cl}(r)$-representations if and only if $\Gamma_{\mc{P}}:=\prod_{J\subset [r]}^{\mc{P}}(\mathrm{Id}_V-\mf{e}_J\mf{f}_J)$ acts as $0$ on $V$.
\begin{proposition}
    \label{prop:4}
    The following are equivalent:
     \begin{enumerate}
    \item The morphism $\hat{e}$ (resp. $\hat{f}$) is bijective.
    \item We have $\Gamma_{\mc{P}}=0$.
     \item  There exists a Clifford algebra $\mathrm{Cl}(r)$ representation on $V$ such that 
    \begin{gather}
      \label{eq:inverse}
      e_I=q_Ie, \quad  f_I=fp_I, \\
        \label{eq3}
        p_{i}e=0, \quad fq_{i}=0, \quad fe=\mathrm{Id}_{W}, \quad ef=p_{[r]}q_{[r]},
      \end{gather}
      where
      \begin{equation*}
      e:=e_{\emptyset}:W\to V, \quad f:=f_{\emptyset}:V\to W.
    \end{equation*}
  \end{enumerate}
\end{proposition}
\begin{proof}
   By direct computations, the morphisms $e_I$ and $f_I$ in \eqref{e9} satisfy
  \begin{align}
    \label{eq:orthogonal}
    f_Je_I=\delta_{IJ}  \quad \text{ for all } \quad I\subset [r], \\
    \label{eq:GammaP2}
    \Gamma_{\mc{P}}=1-\sum_{I\subset [r]}e_If_I.
  \end{align}
By \cref{eq:orthogonal} and \eqref{eq:GammaP2} we have $\hat{f}\hat{e}=\mathrm{Id}$ and  $\hat{e}\hat{f}=\mathrm{Id}-\Gamma_{\mc{P}}$. Hence (1), (2) are equivalent. If (3) holds, then by \cref{cor:faith}
\begin{equation*}
    \sum_{I\subset [r]}e_If_I=\sum_{I\subset [r]}q_{I}p_{[r]}q_{[r]}p_I=1
\end{equation*}
 and hence (2) holds. If (1) and (2) hold, we consider the $V$ as a $\mathrm{Cl}(r)$-representation induced from the isomorphism $\hat{e}$. To verify \cref{eq:inverse} and \cref{eq3}, by \cref{lem:irreducible} we only need to verify the case that $W\cong \mathbb{Z}$ and $V\cong \mathrm{F}(r)$, where $e_I$ maps $1$ to $v_{I}$. It follows from direct computations.
\end{proof}

\begin{remark}
  By \cref{prop:4}, the operators $e_I$ and $f_I$ can be expressed in terms of the orthogonalized operators $e, f, p_i, q_i$, and vice versa.  Specifically, given the orthogonalized operators $e, f, p_i, q_i$, we define $e_I$ and $f_I$ via \eqref{eq:inverse}. Conversely, given the operators $e_I$ and $f_I$ such that \cref{eq:orthogonal} and \cref{eq:GammaP2} hold, by checking the actions on $\mathrm{F}(r)$, the operators $p_i$ and $q_i$ can be written as
\begin{equation*}
p_i := \sum_{I \subseteq [r] \setminus \{i\}} (-1)^{\mathrm{len}(I^{<i})} e_I f_{I \cup \{i\}}, \quad q_i := \sum_{I \subseteq [r] \setminus \{i\}} (-1)^{\mathrm{len}(I^{<i})} e_{I \cup \{i\}} f_I,
\end{equation*}
where $I^{<i} := \{j \in I \mid j < i\}, \quad I^{\geq i} := \{j \in I \mid j \geq i\}$.
\end{remark}

\subsection{The dual representation and the grading}
The isomorphism \cref{eq:ehat} has a duality property: let $e_I,f_I$ satisfy the equations \cref{eq3}. Then the operators 
\begin{equation*}
    \tilde{e}:=e_{[r]}, \quad \tilde{f}:=f_{[r]},\quad \tilde{p}_i:=q_i,\quad \tilde{q}_i:=p_i
\end{equation*}
also satisfy \cref{eq3}.  Moreover, let $\tilde{e}_I:=e_{[r]-I}$. Then $\tilde{e}_I$ also satisfies \cref{eq:inverse} and \cref{eq3}. Hence it induces another isomorphism of $\mathrm{Cl}(r)$-representations
\begin{equation*}
   \hat{\tilde{e}}: W\otimes \mathrm{F}(r)\cong V, \quad w\otimes v_I\mapsto \tilde{e}_Iw.
\end{equation*}
This is called the \textbf{dual representation} of $V$. Moreover, $\hat{\tilde{e}}$ is the composition of $\hat{e}$ with the Hodge star operator:
\begin{equation*}
  *: \mathrm{F}(r)\to \mathrm{F}(r), \quad v_I\mapsto s_I v_{[r]-I}.
\end{equation*}
where $s_I\in \{\pm 1\}$ is a sign depending on $I$. 

Now we consider the graded $\mathrm{Cl}(r)$-representations. The weight and length functions on the subsets $I\subset [r]$ naturally induce the \textbf{length grading} and \textbf{weight grading} on $\mathrm{F}(r)$. With respect to  the Hodge star operator action, we define the \textbf{dual length grading} (resp. \textbf{dual weight grading}) on $\mathrm{F}(r)$ by letting $v_I$ have degree $r-\mathrm{len}(I)$ (resp. $\frac{r(r-1)}{2}-\sum_{i\in I}i$). Given an isomorphism of representations \cref{eq:ehat} with graded $\mathbb{Z}$-modules $W$ and $V$, by direct computation we have the following condition to let \cref{eq:ehat} be an isomorphism of graded $\mathrm{Cl}(r)$-representations.
\begin{lemma}
\label{lem:grading2}
Given an isomorphism of  $\mathrm{Cl}(r)$-representations $\hat{e}$ in \cref{eq:ehat} and gradings on $W$ and $V$, respectively, the following are equivalent:
\begin{enumerate}[leftmargin=*]
  \item The morphism $\hat{e}$ is an isomorphism of graded $\mathrm{Cl}(r)$-representations with respect to the length grading (resp. weight grading) on $\mathrm{F}(r)$;
  \item  The dual representation $\hat{\tilde{e}}$ is an isomorphism of graded $\mathrm{Cl}(r)$-representations with respect to the dual length grading (resp. dual weight grading) on $\mathrm{F}(r)$.
  \item The operators $e,f,p_i,q_i$ shift the degree by $0,0,-1,1$ (resp. $0,0,-i ,i$).
  \item The operators $\tilde{e},\tilde{f},\tilde{p}_i,\tilde{q}_i$ shift the  degree by $r,-r,1,-1$ (resp. $\frac{r(r-1)}{2},-\frac{r(r-1)}{2},i,-i$), 
  \item The operators $e_I$ shift the degree by $\mathrm{len}(I)$ (resp. $\mathrm{wt}(I)$) for all $I\subset [r]$.
  \item The operators $\tilde{e}_I$ shift the degree by $r-\mathrm{len}(I)$ (resp. $\frac{r(r-1)}{2}-\mathrm{wt}(I)$) for all $I\subset [r]$.
\end{enumerate}
\end{lemma}
\begin{remark}
  We will abuse the notation to use $e,f,p_i,q_i$ to denote $\tilde{e},\tilde{f},\tilde{p}_i,\tilde{q}_i$ on the dual representation if there is no confusion.
\end{remark}

\section{The Chow correspondences on Grassmannians of complexes}
\label{sec:chow}
Let $X$ be a smooth variety over $\mathbb{C}$ and $\mc{E}$ be a rank $r$ Tor-amplitude $[0,1]$-perfect complex. Now we state the cohomological and Chow-theoretical version of \cref{eq:once}. We have to impose a $G$-\textbf{smoothness} assumption, to say that $\mc{E}$ is a $G$-smooth complex if for any $d\geq 0$, the Grassmannians $\mathrm{Gr}_X(\mc{E},d)$ and $\mathrm{Gr}_X(\mc{E}^{\vee}[1],d)$ are all smooth varieties (if not empty), and their dimensions are $\dim(X)-d(d-r)$ and $\mathrm{dim}(X)-d(d+r)$ respectively. The complex $\mc{U}_o$ over $\mc{M}_{H}(n)$ under the assumption \cref{assumptionS} is $G$-smooth by \cref{thm:smooth}.

\begin{theorem}
  \label{thm:4.4} Let $\mc{E}$ be a $G$-smooth complex on a smooth variety $X$. If $r=\mathrm{rank}(\mathcal{E})>0$, we define the grading shift for all integers $m$:
    \begin{align*}
    &\mb{CH}_+^{l}(m):= \mathrm{CH}^{l-\frac{1}{2}m(m-1)}(\mathrm{Gr}_X(\mc{E},m)), & \mb{CH}_+:=\bigoplus_{l,m\in \mb{Z}}\mb{CH}_+^{l}(m)\\
     &\mb{CH}_-^{l}(m):= \mathrm{CH}^{l-\frac{1}{2}m(m-1)-mr}(\mathrm{Gr}_X(\mc{E}^{\vee}[1],m)), & \mb{CH}_-:=\bigoplus_{l,m\in \mb{Z}}\mb{CH}_-^{l}(m).
  \end{align*}
Then we have an isomorphism of bigraded $\mathrm{Cl}(r)$-representations
  \begin{equation*}
    \mb{CH}_+ \cong \mb{CH}_-\otimes \mathrm{F}(r)
  \end{equation*}
  where the grading by the rank of quotients on the Grassmannians is compatible with the length grading on $\mathrm{F}(r)$, and the Chow group grading is compatible with the weight grading on $\mathrm{F}(r)$. A similar argument holds for Hodge cohomology if we replace the cohomological grading with the bigrading.
\end{theorem}
In \cref{sec:geom}, we will give the geometric correspondences. Given a class $\alpha\in \mathrm{CH}^*(X)$, let $\int_X \alpha\in \mb{Z}$ be the push-forward of $\alpha$ to a closed point.

\subsection{Geometric correspondences and the Clifford algebra action}
\label{sec:geom}
For every rank $n$ locally free sheaf $E$ on $X$, let $c_i(E)\in \mathrm{CH}^{i}(X)$
be the $i$-th Chern class of $E$ and $c_{top}(E):=c_n(E)$ be the top Chern class. For any partition $\lambda$, we define the Schur class
\begin{equation*}
\Delta_{\lambda}(E):= \mathrm{det}(c_{\lambda_i-i+j}(E))|_{1\leq i,j\leq l} \in \mathrm{CH}^{\mathrm{size}(\lambda)}(X),
\end{equation*}
where $\mathrm{size}(\lambda):=\sum_{i\in \mathbb{Z}}\lambda_i$.
Given a partition $\lambda\in B_{a,b}$ (defined in \cref{sec:cli}), the complement $\lambda^c$ and transpose $\lambda^t$ are defined as 
  \begin{gather*}
    \lambda^c:=(b-\lambda_a,b-\lambda_{a-1},\cdots,b-\lambda_1)\in B_{a,b},\quad 
    \lambda^t:=(\lambda_1^t,\cdots,\lambda_b^t)\in B_{b,a}, 
  \end{gather*}
  where $\lambda_i^t$ is the number of elements in  $\lambda$ with value $\geq i$. We notice that
    \begin{equation*}
    (\lambda^t)^t=\lambda, \quad (\lambda^c)^c=\lambda, \quad (\lambda^c)^t=(\lambda^t)^c.
    \end{equation*}
 We recall the bijection $\iota$ in \cref{eq:iotaa}. For any set $I\subset [r]$, let $ I^{c}:=\{r-1-i\mid i\in I\}$ so that $\iota(I)^c=\iota(I^c)$ for all $I\subset [r]$ where $b=r-d$. Under the bijection $\iota$, the weight function satisfies
\begin{equation}
  \label{eq:weight}
    \mathrm{wt}(I)=\mathrm{size}(\iota(I))+\frac{\mathrm{len}(I)(\mathrm{len}(I)-1)}{2}
\end{equation}
Given integers $a,b\geq 0$, let the partition $(a^{b})$ be $\underbrace{(a,a,\cdots,a)}_{b \text{ times}}$.

Now we define the geometric correspondences. Let $\mc{E}$ be a $G$-smooth complex such that $r:=\mathrm{rank}(\mc{E})>0$. Given integers $d_{-},d_+\geq 0$, we consider the incidence variety as the (classical) fiber product 
\begin{equation*}
    \mathrm{Inc}_{X}(\mc{E},d_-,d_+):=  \mathrm{Gr}_X(\mc{E}^{\vee}[1],d_-)\times_{X}\mathrm{Gr}_X(\mc{E},d_+).
\end{equation*}
By Theorem 1.3 of \cite{zhao2024degeneracy}, $\mathrm{Inc}_{X}(\mc{E},d_-,d_+)$ is a locally complete intersection variety and its dimension is
 \begin{equation}
    \label{lem:excess}
        \mathrm{dim}(\mathrm{Inc}_{X}(\mc{E},d_-,d_+))=\mathrm{dim}(X)-d_+(d_+-r)-d_-(r+d_-)+d_+d_-
    \end{equation}
For any integers $d_-,d_+\geq 0$, let $\mc{U}_+$ be the universal quotient bundle on  $\mathrm{Gr}_X(\mc{E},d_+)$. We consider the  fundamental class of $\mathrm{Inc}_{X}(\mc{E},d_-,d_+)$ in $\mathrm{Gr}_X(\mc{E}^{\vee}[1],d_-)\times \mathrm{Gr}_X(\mc{E},d_+)$:
\begin{equation}
  \label{eq:incidence}
  [\mathrm{Inc}_{X}(\mc{E},d_-,d_+)]\in \mathrm{CH}^{\mathrm{dim}(X)-d_-d_+}(\mathrm{Gr}_X(\mc{E}^{\vee}[1],d_-)\times \mathrm{Gr}_X(\mc{E},d_+)).
\end{equation}
\begin{definition}
    \label{def:correspondence}
  Given $d\geq 0$ and $0\leq j\leq r$, for any subset $I\subset [r]$, we define 
  \begin{align*}
    \Phi^{I}:=\Delta_{\iota(I)^t}(\mc{U}_+)\cap [\mathrm{Inc}_{X}(\mc{E},d-\mathrm{len}(I),d)], 
   \end{align*}
  as a correspondence in  $\mathrm{CH}^{*}(\mathrm{Gr}_X(\mc{E}^{\vee}[1],d-\mathrm{len}(I))\times \mathrm{Gr}_X(\mc{E},d))$. We also define the correspondence $\Psi^I$ as $\Phi^{I^c}$, but shift the order and regard it as a correspondence in $\mathrm{CH}^{*}(\mathrm{Gr}_X(\mc{E},d)\times \mathrm{Gr}_X(\mc{E}^{\vee}[1],d-\mathrm{len}(I)) )$.
\end{definition}

The following lemma follows from direct computations:
\begin{lemma}
    \label{lem:ccccccc}
    For any $I\subset [r]$, the cohomological degree of $\Phi^I$ and $\Psi^I$ are 
    \begin{align*}
        \mathrm{dim}(X)-d(d-\mathrm{len}(I))-\frac{\mathrm{len}(I)(\mathrm{len}(I)-1)}{2}+\mathrm{wt}(I), \\ \mathrm{dim}(X)-d(d-\mathrm{len}(I))+\mathrm{len}(I)r-\frac{\mathrm{len}(I)(\mathrm{len}(I)+1)}{2}-\mathrm{wt}(I).
    \end{align*}
    respectively. Moreover, after the grading shift, $\Psi^I_*$ maps  $\mb{CH}_{+}^{l}(d)$ to $\mb{CH}_{-}^{l-\mathrm{wt}(I)}(d-\mathrm{len}(I))$ and $\Phi^I_*$ maps $\mb{CH}_{-}^{l}(d-\mathrm{len}(I))$ to $\mb{CH}_{+}^{l+\mathrm{wt}(I)}(d)$.
\end{lemma}

The key point to prove \cref{thm:4.4} is the 
following semi-orthogonal relation, which will be proved in \cref{sec:prop:910}.
\begin{proposition}
    \label{prop:910} Let $\mc{P}$ be a total order on $\{I|I\subset [r]\}$ such that if $I\geq_{\mc{P}}J$, then $\mathrm{wt}(I)\geq \mathrm{wt}(J)$. Then
  \begin{align}
\label{rule1c}
  \Psi^I \circ \Phi^J =\delta_{IJ} \operatorname{Id} \text{ if } I\geq_{\mc{P}}J. \\
    \label{rule2c} \prod_{I\subset [r]}^{\mc{P}}(1-\Phi^{I}\Psi^{I}) =0
  \end{align}
\end{proposition}

\begin{proof}[Proof of \cref{thm:4.4}]
We let $e_I$ and $f_I$ be the orthogonalization of $\Psi_*^I$ and $\Phi_*^I$ for all $I\subset [r]$. Then \cref{thm:4.4} directly follows from \cref{prop:4} and \cref{prop:910}. The compatibility of gradings follows from \cref{lem:grading2} and \cref{lem:ccccccc}.

For Hodge cohomology groups, we note that the cycle class map $\mathrm{Hdg}:\mathrm{CH}^p(X)\to H^{p,p}(X,\mb{Q})$ for smooth projective varieties is compatible with the correspondences. Hence we obtain \cref{thm:4.4} for Hodge cohomology groups.
\end{proof}

  \begin{remark}
    \label{rem:dualrep}
    We should remark that we will use the dual representation on $\mb{CH}_+$ more often. The reason is that for the dual representation, the operators \begin{equation*}
    e:\mb{CH}_-^{l}(m)\to \mb{CH}_+^{l+\frac{1}{2}r(r-1)}(m+r), \quad f:\mb{CH}_+^{l+\frac{1}{2}r(r-1)}(m+r)\to \mb{CH}_{-}^{l}(m)
  \end{equation*}
 are correspondences between birational varieties $\mathrm{Gr}_X(\mc{E}^\vee[1],m)$ and $\mathrm{Gr}_X(\mc{E},m+r)$ induced by the fundamental class
  \begin{equation*}
      [\mathrm{Ind}_X(\mc{E},m,m+r)]\in \mathrm{CH}^*(\mathrm{Gr}_X(\mc{E}^\vee[1],m)\times \mathrm{Gr}_X(\mc{E},m+r)).
  \end{equation*}
\end{remark}
\subsection{The semi-orthogonal relations}
\label{sec:prop:910} The proof of \cref{prop:910} is similar to \cite
{nakajima1997heisenberg}, where the key point is a dimension estimate on the incidence varieties. First, we recall the relation between the $G$-smooth property and the degeneracy loci:
\begin{lemma}[Part of Theorem 1.3 of \cite{zhao2024degeneracy}]
     \label{thm:expected} If $\mc{E}$ is a $G$-smooth complex of rank $r$, for any $d\in \mb{Z}$, we consider the $d$-th (Fitting) degeneracy locus
  \begin{equation*}
    X^{\geq d}(\mc{E}):=\{x\in X \mid \mathrm{rank}(h^{0}(\mc{E}_x))\geq d\}, \quad X^{=d}(\mc{E}):=X^{\geq d}(\mc{E})\backslash X^{\geq d+1}(\mc{E}).
  \end{equation*}
Then for any $d\geq 0$,  $X^{=d}(\mc{E})$ is a locally complete intersection variety and 
  \begin{equation*}
          \min\{\mathrm{dim}(X)-d(d-r),\mathrm{dim}(X)\}=\mathrm{dim}(X^{=d}(\mc{E}))= \mathrm{dim}(X^{\geq d}(\mc{E})).
  \end{equation*}
\end{lemma}
\begin{corollary}
  \label{cor:nonempty}
  If $\mc{E}$ is a $G$-smooth complex of rank $r$, then for any $d\geq \max\{r, 0\}$, $X^{\geq d}(\mc{E})$ is non-empty if and only if $X^{=d}(\mc{E})$ is non-empty.
\end{corollary}
\begin{proof}
  We only need to notice that when $d\geq \max\{r, 0\}$, $\mathrm{dim}(X^{\geq d}(\mc{E}))=\mathrm{dim}(X)-d(d-r)$ is a strictly decreasing function of $d$.
\end{proof}

\begin{proof}[Proof of \cref{prop:910}]
  Assuming \cref{rule1c}, to prove  \cref{rule2c}, we first prove that
    \begin{align}
    \label{rule2} \prod_{I\subset [r]}^{\mc{P}}(1-\Phi^{I}_*\Psi^{I}_*) =0.
  \end{align}
  By \cref{prop:4}, we only need to know that the morphism 
 \begin{equation*}
   \bigoplus_{j=0}^r\bigoplus_{I\subset [r],\mathrm{len}(I)=j}\mb{CH}_-^{l-\mathrm{wt}(I)}(d-j)\xrightarrow{\Phi^I_*}\mb{CH}_+^{l}(d)
    \end{equation*}
 is an isomorphism for any $d\in \mb{Z}$. It is precisely Theorem 1.1 in \cite{jiang2020chow}. To lift \cref{rule2} to \cref{rule2c}, we notice that for any smooth variety $T$, $\mc{E}_T$ is $G$-smooth over $X\times T$, where $\mc{E}_T$ is the pullback of $\mc{E}$ from $X$ to $X\times T$. Hence \cref{rule2c} follows from \cref{rule2} and Manin's identity principle (\cref{thm:manin}).

 Now we prove \cref{rule1c}.  We assume that $\mathrm{wt}(I)\geq \mathrm{wt}(J)$. Let $d_1:=d-\mathrm{len}(J),d_2:=d-\mathrm{len}(I)$. By \cref{def:correspondence} and direct computation, the homological degree of the correspondence $\Psi^I\Phi^J$ is 
 \begin{equation*}
    \beta_1:=\dim(X)-\frac{d_1^2+d_2^2}{2}+\frac{d_1-d_2}{2}-d_1r+\mathrm{wt}(J)-\mathrm{wt}(I). 
 \end{equation*}
On the other hand, we prove that $\dim(\mathrm{Gr}_{X}(\mc{E}^{\vee}[1],d_1)\times_X \mathrm{Gr}_{X}(\mc{E}^{\vee}[1],d_2))$ is
  \begin{equation}
    \label{eq:dimform}
  \beta_2:=\mathrm{dim}(X)-d_2^2-d_1^2+d_2d_1-\max\{d_1,d_2\}r.
  \end{equation} 
if the fiber product is not empty. By \cref{cor:nonempty}, we can assume that the variety $X^{=\max\{d_1,d_2\}}(\mc{E}^{\vee}[1])$ is non-empty. We consider the projection
  \begin{equation*}
    \mathrm{Gr}_{X}(\mc{E}^{\vee}[1],d_1)\times_X \mathrm{Gr}_{X}(\mc{E}^{\vee}[1],d_2)\to X.
    \end{equation*}
The fiber over $X^{=k}(\mc{E})$ is not empty only if $k\geq \max\{d_1,d_2\}$, where it is a $\mathrm{Gr}(k,d_1)\times \mathrm{Gr}(k,d_2)$-bundle, and hence by \cref{thm:expected} has dimension
    \begin{align}
        \label{eq:dimcalc}
     \mathrm{dim}(X)-k(k+r)+d_1(k-d_1)+d_2(k-d_2)
    \end{align}
    which is a strictly decreasing function on $k$ when $k\geq \max\{d_1,d_2\}$. Hence the maximum dimension is attained when $k=\max\{d_1,d_2\}$, and \cref{eq:dimform} is proved. We notice that $\beta_1-\beta_2$ is 
\begin{equation*}
    \frac{(d_1-d_2)^2}{2}+\frac{(d_1-d_2)}{2}-(d_1-\max\{d_1,d_2\})r+\mathrm{wt}(J)-\mathrm{wt}(I)\geq 0,
\end{equation*}
and the equality holds if and only if $d_1=d_2$ and $\mathrm{wt}(I)=\mathrm{wt}(J)$. Hence we have $\Psi^I\Phi^J=0$ unless $\mathrm{len}(I)=\mathrm{len}(J)$ and $\mathrm{wt}(I)=\mathrm{wt}(J)$.

Now we assume that $\mathrm{len}(I)=\mathrm{len}(J)$ and $\mathrm{wt}(I)=\mathrm{wt}(J)$. Still by \cref{eq:dimcalc}, we notice that the image of the diagonal map \begin{equation*}
  \Delta_{d_1}:\mathrm{Gr}_{X}(\mc{E}^{\vee}[1],d_1)\to \mathrm{Gr}_{X}(\mc{E}^{\vee}[1],d_1)\times_X \mathrm{Gr}_{X}(\mc{E}^{\vee}[1],d_1)
\end{equation*}
is the only irreducible component which has dimension $\beta_2$. Hence there exists an integer $a_{IJ}$ such that 
\begin{equation*}
    \Psi^I\Phi^J=a_{IJ}\mathrm{Id}.
\end{equation*}
Finally we compute the constant $a_{IJ}$. Let $\lambda=\iota(J)^t$ and $\mu=\iota(I)^t$, where $\lambda,\mu\in B_{r-d+d_1,d-d_1}$.  
We can assume that $X^{\geq d_1}(\mc{E}^{\vee}[1])\neq \emptyset$. Otherwise, we have $\mathrm{Gr}_X(\mc{E}^{\vee}[1],d_1)=\emptyset$ and $a_{IJ}$ can be any integer. By \cref{cor:nonempty}, we can choose a closed point $x\in X^{=d_1}(\mc{E}^{\vee}[1])=X^{=(d_1+r)}(\mc{E})$. To compute $a_{IJ}$, we can replace $X$ by any open neighborhood $U\subset X$, as $\mc{E}|_{U}$ is also $G$-smooth over $U$. It is well-known (we refer to Proposition 22.53 of \cite{gortz2020algebraic}) that there is a neighborhood $U$ of $x$ such that $\mc{E}|_U$ is quasi-isomorphic to a two-term complex 
  \begin{equation*}
    \mc{O}_U^{d_1}\xrightarrow{\beta} \mc{O}_U^{d_1+r},
  \end{equation*}
where $\beta$ is a $d_1\times (d_1+r)$ matrix over $\mc{O}(U)$ such that $\beta|_x=0$. Thus we have a morphism from $U$ to the affine space $\mathrm{Hom}(\mb{C}^{d_1},\mb{C}^{d_1+r})$ such that $\mc{E}|_{U}$ is the pull-back of the universal complex 
  \begin{equation*}
    \mc{E}_{r,d_1}:\mc{O}^{d_1}\to \mc{O}^{r+d_1}
  \end{equation*}
 from $\mathrm{Hom}(\mb{C}^{d_1},\mb{C}^{d_1+r})$ to $U$. Moreover, by dimension counting in \cref{lem:excess}, there are no excess bundles for the Cartesian diagrams:
\begin{equation*}
    \begin{tikzcd}
\mathrm{Inc}_U(\mc{E}|_{U},d_-,d_+) \ar{r}\ar{d}& \mathrm{Inc}_{\mathrm{Hom}(\mb{C}^{d_1},\mb{C}^{d_1+r})}(\mc{E}_{r,d_1},d_-,d_+)\ar{d} \\
U\ar{r} & \mathrm{Hom}(\mb{C}^{d_1},\mb{C}^{d_1+r})
    \end{tikzcd}
\end{equation*}
for all $d_-,d_+\geq 0$. Thus, to compute the multiplicity $a_{IJ}$, we can replace $U$ by the affine space  $\mb{V}_{r,d_1}:=\mathrm{Hom}(\mb{C}^{d_1},\mb{C}^{r+d_1})$. In this case, we have
 \begin{gather*}
  \mathrm{Gr}_{\mb{V}_{r,d_1}}(\mc{E}^{\vee}[1],d_1)=\{0\}, \quad  \mathrm{Gr}_{\mb{V}_{r,d_1}}(\mc{E},d)\cong \mathrm{Tot}_{\mathrm{Gr}(d_1+r,d)}(\mc{S}^{\oplus d_1}),\\
   \mathrm{Inc}_{\mb{V}_{r,d_1}}(\mc{E},d_1,d)\cong \mathrm{Gr}(d_1+r,d),
\end{gather*}
where $\mc{S}$ is the universal subbundle on $\mathrm{Gr}(d_1+r,d)$:
\begin{equation*}
  0\to\mc{S}\to \mc{O}^{r+d_1}\to \mc{U}\to 0.
\end{equation*}
The projection map from $\mathrm{Inc}_{\mb{V}_{r,d_1}}(\mc{E},d_1,d)$ to $\mathrm{Gr}_{\mb{V}_{r,d_1}}(\mc{E},d)$ is the inclusion of the zero section of the total space.  By the excess intersection formula and the Schubert calculus, we have
\begin{align*}
    a_{IJ}&=\int_{\mathrm{Gr}(d_1+r,d)}(\Delta_{\lambda}(\mc{U})\Delta_{\mu^c}(\mc{U})\cdot c_{top}(\mc{O}^{d_1}\otimes \mc{S}^{\vee}))\\
    &= \int_{\mathrm{Gr}(d_1+r,d)}(\Delta_{\lambda^t}(\mc{S}^{\vee})\Delta_{\mu^{ct}}(\mc{S}^{\vee})\cdot \Delta_{(r+d_1-d)^{d_1}}(\mc{S}^{\vee})), \\
    &=\int_{\mathrm{Gr}(r,d-d_1)}(\Delta_{\lambda^t}(\mc{S}^{\vee})\Delta_{\mu^{ct}}(\mc{S}^{\vee}))=\delta_{\lambda^t,\mu^t}=\delta_{IJ}.
\end{align*}
\end{proof}
\begin{theorem}[Manin's identity principle \cite{manin1968correspondences}]\label{thm:manin}
  Let $X$ and $Y$ be two smooth projective varieties, and $f\in \mathrm{CH}^*(X\times Y)$ be a correspondence. Then $f=0$ if and only if for any smooth variety $T$, $(f\times id)_*=0$ for the correspondence
  \begin{equation*}
    (f\times \operatorname{id})_*:\mathrm{CH}^*(X\times T)\to \mathrm{CH}^*(Y\times T).
  \end{equation*}
\end{theorem}

\section{Derived Categories of Derived Grassmannians}
\label{sec:derived}
In this section, we prove the Grothendieck group and Hochschild homology versions of \cref{eq:once}. The method is different from \cref{sec:chow}. The semiorthogonal decomposition of $\mathrm{D}^b_{\mathrm{coh}}(\mathrm{Gr}_{X}(\mc{E}))$ established by the first-named author in \cite{jiang2023derived} induces several \textbf{categorical equations}. By considering their decategorification, we uncover the underlying representation-theoretic structures through the lens of Clifford algebra representations.

In this section, we apply the derived algebraic geometry framework. All the categories in this section are $\infty$-categories and all the schemes in this section are derived schemes.

\subsection{Semiorthogonal decompositions and categorical equations}
We begin by describing how semiorthogonal decompositions can be expressed in terms of categorical equations.

A \textbf{semiorthogonal decomposition} of a stable $\infty$-category $\mathcal{D}$, denoted by
\[
\mathcal{D} = \langle \mathcal{D}_1, \mathcal{D}_2, \ldots, \mathcal{D}_n \rangle,
\]
is a sequence of stable $\infty$-subcategories $\mathcal{D}_1, \ldots, \mathcal{D}_n$ satisfying:
\begin{enumerate}
    \item $\operatorname{Hom}(D_j, D_i) \cong 0$ for all $D_i \in \mathcal{D}_i$, $D_j \in \mathcal{D}_j$, whenever $j > i$.
    \item For any $D \in \mathcal{D}$, there is a sequence
    \[
    0 = D_n \to D_{n-1} \to \cdots \to D_1 \to D_0 = D
    \]
    in $\mathcal{D}$ such that $\mathrm{cone}(D_j \to D_{j-1}) \in \mathcal{D}_j$ for each $1 \leq j \leq n$.
\end{enumerate}
The sequence $\mathcal{D}_1, \ldots, \mathcal{D}_n$ is called \textbf{semiorthogonal} if only condition (1) is required.

Now, let $\mathcal{A}_1, \ldots, \mathcal{A}_n$ be stable $\infty$-categories and consider functors
\[
\Phi_j \colon \mathcal{A}_j \to \mathcal{D}.
\]
Suppose that each $\Phi_j$ admits a left adjoint $\Phi_j^L$. The adjunction yields natural unit and counit maps:
\[
\operatorname{id}_{\mathcal{D}} \to \Phi_j \Phi_j^L, \qquad \Phi_j^L \Phi_j \to \operatorname{id}_{\mathcal{A}_j}.
\]
Set $\mathfrak{R}_j := \mathrm{cone}\big(\operatorname{id}_{\mathcal{D}} \to \Phi_j \Phi_j^L\big)[-1]$, the homotopy fiber of the unit map.

We now reformulate the defining conditions for a semiorthogonal decomposition in terms of the following \textbf{categorical equations}:
\begin{align}
   \label{cat1} \mathrm{cone}(\Phi_j^L\Phi_j \to \operatorname{id}_{\mathcal{A}_j}) \cong 0 \\
   \label{cat2} \Phi_i^L \circ \Phi_j \cong 0 \text{ for all } j>i\\
   \label{cat3} \mathfrak{R}_n \circ \cdots \circ \mathfrak{R}_1 \cong 0
\end{align}
\begin{proposition}
\label{lem:515}
    For each $j$, the functor $\Phi_j$ is fully faithful if and only if \cref{cat1} holds. If all $\Phi_j$ are fully faithful, the sequence of essential images 
      \[
    \Phi_1(\mathcal{A}_1),\ \Phi_2(\mathcal{A}_2),\ \ldots,\ \Phi_n(\mathcal{A}_n)
    \]
    is semiorthogonal if and only if \cref{cat2} holds. Moreover, if the above conditions are satisfied, the sequence of essential images of $\Phi_j$ forms a semiorthogonal decomposition of $\mathcal{D}$ if and only if \cref{cat3} holds.
\end{proposition}

\begin{proof}
These statements are implicitly proved in \cite[Theorem 3.2]{jiang2023derived}. Specifically, the first two statements follow from adjunction and the definition of semiorthogonality. For the last statement, set $\mathfrak{Fil}_{0} = \operatorname{id} \colon \mathcal{D} \to \mathcal{D}$. For each $1 \leq j \leq n$, define:
\begin{align*}
  \mathfrak{Fil}_{j} := \mathfrak{R}_j \circ \mathfrak{R}_{j-1} \circ \cdots \circ \mathfrak{R}_1 \colon \mathcal{D} \to \mathcal{D},  \quad \mathfrak{pr}_{j} := \Phi_j^L \circ \mathfrak{R}_{j-1} \circ \cdots \circ \mathfrak{R}_1 \colon \mathcal{D} \to \mathcal{A}_j.
\end{align*}
These fit into a canonical diagram in the stable $\infty$-category $\mathrm{Fun}(\mathcal{D}, \mathcal{D})$:
\[
\begin{tikzcd}[back line/.style={dashed}, row sep=1.8em, column sep=0.8em]
 \mathfrak{Fil}_{n} \ar{rr} & & \mathfrak{Fil}_{n-1} \ar{rr} \ar{ld} & & \cdots \ar{ld} \ar{r} & \mathfrak{Fil}_{1} \ar{rr} & & \mathfrak{Fil}_{0} := \operatorname{id} \ar{ld} \\
& \Phi_{n} \circ \mathfrak{pr}_n \ar[dashed]{lu} & & \Phi_{n-1} \circ \mathfrak{pr}_{n-1} \ar[dashed]{lu} & & & \Phi_{1} \circ \mathfrak{pr}_1, \ar[dashed]{lu}
\end{tikzcd}
\]
where each triangle is an exact triangle. Thus, the functors $\Phi_j$ induce a semiorthogonal decomposition of $\mathcal{D}$ if and only if $\mathfrak{Fil}_n \cong 0$.
\end{proof}

\begin{remark}
Dually, suppose that each $\Phi_j$ admits a right adjoint $\Phi_j^R$ (and not necessarily a left adjoint),
define $\mathfrak{L}_j := \mathrm{cone}(\Phi_j \Phi_j^R \to \operatorname{id})$, the cone of the counit. Then, by a similar argument, the lemma holds with \cref{cat1,cat2,cat3} replaced by:
\begin{align}
   \mathrm{cone}(\operatorname{id}_{\mathcal{A}_j} \to \Phi_j^R\Phi_j) \cong 0, \tag{\ref{cat1}$'$} \\
   \Phi_j^R \circ \Phi_i \cong 0 \quad \text{for all } j > i, \tag{\ref{cat2}$'$} \\
   \mathfrak{L}_1 \circ \cdots \circ \mathfrak{L}_n \cong 0. \tag{\ref{cat3}$'$}
\end{align}
\end{remark}

\begin{remark}
  For any integer $n\geq 1$, we consider the unital algebra generated by $m_i,\psi_i,\phi_i$ for $1\leq i\leq n$ and $m_{\infty}$ with relations
  \begin{gather*}
     m_im_{j}=\delta_{ij}m_i, \quad
    \phi_i m_j=\delta_{ij}\phi_i, \quad m_i \psi_j=\delta_{ij}\psi_j, \text{ for all } i,j\in [1,n]\cup \{\infty\} \\
    m_{\infty}\phi_i=\phi_i, \quad \psi_i m_{\infty}=\psi_i, \quad  m_{j}\phi_i=0, \quad \psi_i m_j=0 \text{ for all }1\leq i,j\leq n\\
    \psi_{i}\phi_{j}=\delta_{ij}m_j, \quad \text{ for } 1\leq i\leq j\leq n, \\
\sum_{i=1}^n m_i=(1-\phi_1\psi_1)\cdots (1-\phi_n\psi_n)=1-m_{\infty}.
  \end{gather*}
  Then by \cref{lem:515}, any semi-orthogonal decomposition of a stable $\infty$-category $\mc{D}$ into $n$ components can be formulated as a categorical representation of the above algebra, where $m_i$ acts as the projection to the $i$-th component, and $m_{\infty}$ acts as the projection to $\mc{D}$. 
\end{remark}
\subsection{Derived Grassmannians of perfect complexes}
Next, we focus on derived Grassmannians of perfect complexes and their incidence correspondences \cite{jiang2022grassmanian, jiang2023derived}.

Let $X$ be a scheme (or, more generally, a derived stack), and let $\mathcal{E}$ be a perfect complex on $X$ of Tor amplitude in $[0,1]$ and rank $r \ge 0$. For an integer $d_+ \geq 0$, we consider the (rank $d_+$) \textbf{derived Grassmannian}
\begin{equation*}
 \mathrm{Gr}_X(\mathcal{E},d_+) \to X
\end{equation*}
of the complex $\mathcal{E}$ over $X$ as defined in \cite{jiang2022grassmanian}. It assigns to any morphism $T \to X$ (with $T$ a derived scheme) the space of quotients $\rho_+ \colon \mathcal{E}_T \twoheadrightarrow \mathcal{Q}_+$, where $\mathcal{Q}_+$ is a vector bundle on $T$ of rank $d_+$, and $\rho_+$ induces a surjection on zeroth homology. Here, $\mathcal{E}_T$ denotes the derived pullback of $\mathcal{E}$ to $T$. It is a derived enhancement of the classical Grassmannian of quotients of the sheaf $h^0(\mathcal{E})$, i.e. they have the same scheme structure, but different derived structures.

By convention, we set $\mathrm{Gr}_X(\mathcal{E}, 0) = X$ and $\mathrm{Gr}_X(\mathcal{E}, d_+) = \emptyset$ for $d_+ < 0$. The projection $\mathrm{Gr}_X(\mathcal{E},d_+) \to X$ is representable, proper, and quasi-smooth, with virtual relative dimension $d_+(r - d_+)$ (see \cite[Corollary 4.47]{jiang2022grassmanian}).

The shifted derived dual $\mathcal{E}^{\vee}[1]$ is also a perfect complex of Tor amplitude $[0,1]$, now with rank $-r$. For any $d_- \in \mathbb{Z}$, we similarly have the derived Grassmannian
\begin{equation*}
    \mathrm{Gr}_X(\mathcal{E}^{\vee}[1],d_-) \to X,
\end{equation*}
which is proper and quasi-smooth, with virtual relative dimension $-d_-(r + d_-)$.

Given $d_+, d_- \in \mathbb{Z}$, the \textbf{derived incidence scheme} defined in \cite[\S2.2]{jiang2022grassmanian},
\[
\mathrm{Inc}_X(\mathcal{E},d_+,d_-) \to X,
\]
has the same underlying classical scheme as the derived fiber product
\begin{equation*}
    \mathrm{Gr}_X(\mathcal{E}^{\vee}[1],d_-) \times_X \mathrm{Gr}_X(\mathcal{E},d_+),
\end{equation*}
but with a \textbf{different} derived structure. Explicitly, for any $T \to X$ (with $T$ a derived scheme), the space of $T$-points of $\mathrm{Inc}_X(\mathcal{E},d_+,d_-)$ is given by triples
\[
\Big(\rho_+ \colon  \mathcal{E}_T \twoheadrightarrow \mathcal{Q}_+,\, \rho_- \colon \mathcal{E}_T^\vee[1] \twoheadrightarrow \mathcal{Q}_-,\, \gamma\Big),
\]
where $\rho_+$ and $\rho_-$ are $T$-points of $\mathrm{Gr}_X(\mathcal{E},d_+)$ and $\mathrm{Gr}_X(\mathcal{E}^\vee[1], d_-)$, respectively, and $\gamma$ is a homotopy between the composition
\[
\mathcal{Q}_-^\vee[1] \xrightarrow{\rho_-^\vee[1]} \mathcal{E}_T \xrightarrow{\rho_+} \mathcal{Q}_+
\]
and zero. The ``convolution" of the composition yields a canonical perfect complex
\[
\mathrm{cone}\Big( \mathrm{cone} \big( \mathcal{Q}_-^\vee[1] \xrightarrow{\rho_-^\vee[1]} \mathcal{E}_T\big)  \xrightarrow{\rho_+}  \mathcal{Q}_+ \Big)[-1] \in \mathrm{D}_{\mathrm{perf}}(T).
\]
This defines a \textbf{universal perfect complex}
\[
\mathcal{E}^{\mathrm{univ}}_{d_+, d_-} \in  \mathrm{D}_{\mathrm{perf}}(\mathrm{Inc}_X(\mathcal{E},d_+,d_-))
\]
on $\mathrm{Inc}_X(\mathcal{E},d_+,d_-)$ of Tor amplitude in $[0,1]$ and rank $r-d_+ + d_-$.

By construction, we have a correspondence diagram
\[
\begin{tikzcd}
& \mathrm{Inc}_X(\mathcal{E}, d_+, d_- ) \ar{dl}[swap]{r_+} \ar{dr}{r_-} & \\
\mathrm{Gr}_X(\mathcal{E}, d_+) & & \mathrm{Gr}_X(\mathcal{E}^\vee[1], d_-)
\end{tikzcd}
\]
where the projections $r_+$, $r_-$, and $\mathrm{Inc}_X(\mathcal{E},d_+,d_-) \to X$ are proper and quasi-smooth, with virtual relative dimensions
\begin{align*}
    d_-(d_+ - r - d_-), \quad d_+(r + d_- - d_+), \quad \text{and} \quad d_+(r - d_+) - d_-(r + d_-) + d_-d_+,
\end{align*}
respectively (\cite[Lemma 2.10]{jiang2023derived}).

\subsection{Semiorthogonal decompositions and Clifford algebra representations}
We now relate the semiorthogonal decompositions of derived Grassmannians to Clifford algebra representations, completing the construction of the Grothendieck group and Hochschild homology version of \cref{eq:once}.

Throughout, for any scheme or derived stack $Y$, we let $\mathrm{D}(Y)$ denote one of the following derived categories:
\begin{equation*}
  \mathrm{D}_{\mathrm{qc}}(Y), \qquad \mathrm{D}^{\mathrm{b}}_{\mathrm{coh}}(Y), \qquad \mathrm{D}_{\mathrm{perf}}(Y),
\end{equation*}
namely, the derived category of quasi-coherent sheaves, bounded coherent complexes, or perfect complexes, respectively.

Continuing with the previous subsection, we consider the unions
\[
\mathrm{Gr}_X(\mathcal{E}, \bullet) = \coprod_{d_+ \in \mathbb{Z}} \mathrm{Gr}_X(\mathcal{E}, d_+), \qquad
\mathrm{Gr}_X(\mathcal{E}^\vee[1], \bullet) = \coprod_{d_- \in \mathbb{Z}} \mathrm{Gr}_X(\mathcal{E}^\vee[1], d_-).
\]
and their corresponding derived categories:
\[
\mathrm{D}(\mathrm{Gr}_X(\mathcal{E}, \bullet)) = \bigoplus_{d_+ \in \mathbb{Z}} \mathrm{D}(\mathrm{Gr}_X(\mathcal{E}, d_+)), 
\quad 
\mathrm{D}(\mathrm{Gr}_X(\mathcal{E}^\vee[1], \bullet)) = \bigoplus_{d_- \in \mathbb{Z}} \mathrm{D}(\mathrm{Gr}_X(\mathcal{E}^\vee[1], d_-)).
\]
For each $i \in \mathbb{Z}$, we consider the incidence correspondence diagram
\[
\begin{tikzcd}
& \mathrm{Inc}_X(\mathcal{E}, i) := \coprod_{d \in \mathbb{Z}} \mathrm{Inc}_X(\mathcal{E}, d+ i, d) \ar{dl}[swap]{r_+} \ar{dr}{r_-} & \\
\mathrm{Gr}_X(\mathcal{E}, \bullet + i) & & \mathrm{Gr}_X(\mathcal{E}^\vee[1], \bullet)
\end{tikzcd}
\]
and the universal perfect complex
\[
\mathcal{E}^{\mathrm{univ}}_i = \big(\mathcal{E}^{\mathrm{univ}}_{d + i, d}\big)_{d \in \mathbb{Z}} \in \mathrm{D}_{\mathrm{perf}}(\mathrm{Inc}_X(\mathcal{E}, i))
\]
which has Tor amplitude in $[0,1]$ and rank $r-i$ over $\mathrm{Inc}_X(\mathcal{E}, i)$.

For any subset $I \subset [r]$,  write $I = (i_0 < i_1 < \cdots < i_{j-1})$ with $\mathrm{len}(I)=j$. We recall the associated partition
\[
\iota(I) := (i_{j-1} - (j-1) \geq  \cdots \geq i_1 - 1 \geq i_0) \in B_{j, r-j}
\]
defined in \cref{eq:iotaa}. Consider the Fourier--Mukai functor
\[
\Phi^I \colon \mathrm{D}(\mathrm{Gr}_X(\mathcal{E}^\vee[1], \bullet) \to \mathrm{D}(\mathrm{Gr}_X(\mathcal{E}, \bullet))
\]
which sends an object $\mathcal{F} \in \mathrm{D}(\mathrm{Gr}_X(\mathcal{E}^\vee[1], d -r +j))$ to 
\[
\Phi^I(\mathcal{F}) =
r_{+*}\left(r_-^* (\mathcal{F}) \otimes \mathbb{S}^{\iota(I)}\big(\mathcal{E}_{r-j}^{\mathrm{univ}}\big)\right) \otimes \det(\mathcal{U}_+)^{r-j} \in  \mathrm{D}(\mathrm{Gr}_X(\mathcal{E}, d)),
\]
where $\mathbb{S}^{\iota(I)}(-)$ denotes the derived Schur functor (\cite{jiang2022grassmanian}), $\mathcal{E}_{r-j}^{\mathrm{univ}}$ is the universal perfect complex on $\mathrm{Inc}_X(\mathcal{E},r-j)$, and $\mathcal{U}_+$ is the universal quotient bundle on $\mathrm{Gr}_X(\mathcal{E},d)$. 

By \cite[Lemma 2.11]{jiang2023derived}, the functor $\Phi^I$ admits a left adjoint
\[
\Psi^I \colon  \mathrm{D}(\mathrm{Gr}_X(\mathcal{E}, \bullet)) \to \mathrm{D}(\mathrm{Gr}_X(\mathcal{E}^\vee[1], \bullet)),
\]
which is explicitly given by a Fourier--Mukai kernel supported on $\mathrm{Inc}_X(\mathcal{E}, r-j)$.

There is a natural lexicographic order on subsets of $[r]$: for $I = \{i_0 < \cdots < i_{j-1}\}$ and $J = \{k_0 < \cdots < k_{\ell-1}\}$, we write $I <_{\mathrm{lex}} J$ if the sequence $(i_{j-1}, \ldots, i_0, -1, -1, \ldots)$ precedes $(k_{\ell-1}, \ldots, k_0, -1, -1, \ldots)$ in the usual lexicographic order. Under the bijection $\iota: \{I|I\subset [r]\} \to \mathrm{P}_r$ from \cref{eq:iotaa}, this lexicographic order $<_{\mathrm{lex}}$ is precisely the reverse of the total order $<$ on $\mathrm{P}_r$ defined in \cite[Notation 3.1]{jiang2023derived}, i.e., $I <_{\mathrm{lex}} J$ if and only if $\iota(I) > \iota(J)$ in $\mathrm{P}_r$. Therefore, we have:

\begin{theorem}[{Cf. \cite[Theorem 3.2]{jiang2023derived}}]
\label{thm:SOD}
With notation as above:
\begin{enumerate}
     \item For any $I \subset [r]$, we have $\mathrm{cone}(\Psi^I \Phi^I \to \operatorname{id}) \cong 0$.
    \item For any $I, J \subset [r]$, we have $\Psi^I \circ \Phi^J \cong 0$ whenever $J <_{\mathrm{lex}} I$.
    \item Setting $\mathfrak{R}_I := \mathrm{cone}(\operatorname{id} \to \Phi^I \Psi^I)[-1]$, we have $\prod_{I \subset [r]}^{<_{\mathrm{lex}}} \mathfrak{R}_I \cong 0$.
\end{enumerate}
In particular, all $\Phi^I$ are fully faithful, and induce a semiorthogonal decomposition 
\[
\mathrm{D}(\mathrm{Gr}_X(\mathcal{E}, \bullet)) =  \left \langle \mathrm{Im} (\Phi^I) \mid I \subset [r] \right \rangle
\]
which is ordered by $>_{\mathrm{lex}}$, i.e., $\operatorname{Hom}(\mathrm{Im}(\Phi^J),\, \mathrm{Im}(\Phi^I)) \cong 0$ whenever $J <_{\mathrm{lex}} I$.
\end{theorem}
Consequently, for each $d \in \mathbb{Z}$, the fully faithful functors $\Phi^I$ induce a semiorthogonal decomposition of the degree-$d$ component:
\[
\mathrm{D}(\mathrm{Gr}_X(\mathcal{E}, d)) =
\left\langle \Phi^I\left(\mathrm{D}\big(\mathrm{Gr}_X(\mathcal{E}^\vee[1],\, d - r + \mathrm{len}(I))\big)\right) \;\middle|\; I \subset [r] \right\rangle.
\]

\begin{theorem}
  \label{thm:1} 
  We have an isomorphism of graded $\mathrm{Cl}(r)$-representations given by $\mathrm{G}_0$-theoretic correspondences
  \begin{equation*}
    \bigoplus_{m\in \mathbb{Z}}\mathrm{G}_0(\mathrm{Gr}_X(\mc{E},m))\cong \bigoplus_{m\in \mb{Z}}\mathrm{G}_0(\mathrm{Gr}_X(\mc{E}^{\vee}[1],m))\otimes \mathrm{F}(r)
  \end{equation*}
 where $\mathrm{F}(r)$ is equipped with the dual length grading. The same statement applies to $\mathrm{K}_0$. 
\end{theorem}
\begin{proof}
In the setting of \cref{thm:SOD}, consider the morphisms induced on Grothendieck groups (where $\mathrm{D}$ denotes either $\mathrm{D}^\mathrm{b}_{\mathrm{coh}}$ or $\mathrm{D}_{\mathrm{perf}}$):
\begin{align*}
    \mf{e}_I &:= [\Phi^I] \colon K_0(\mathrm{D}(\mathrm{Gr}_X(\mathcal{E}, \bullet))) \to K_0(\mathrm{D}(\mathrm{Gr}_X(\mathcal{E}^\vee[1], \bullet))), \\
    \mf{f}_I &:= [\Psi^I] \colon K_0(\mathrm{D}(\mathrm{Gr}_X(\mathcal{E}^\vee[1], \bullet))) \to K_0(\mathrm{D}(\mathrm{Gr}_X(\mathcal{E}, \bullet))).
\end{align*}
By \cref{thm:SOD}, the morphisms $\mf{e}_I$ and $\mf{f}_I$ for all $I\subset [r]$ are semi-orthogonal with the lexicographic order $\mathcal{P}={\mathrm{lex}}$. Let $e_I$ and $f_I$ be the orthogonalization of  $\{\mf{e}_I\}$ and $\{\mf{f}_I\}$. Then \cref{thm:1} directly follows from (3) of \cref{thm:SOD} and \cref{prop:910}. The compatibility of gradings follows from  \cref{lem:grading2}.
\end{proof}

\begin{remark}
Given a derived scheme, by Proposition 3.3 of \cite{Adeel}, it has the same Grothendieck group of coherent sheaves as its underlying classical scheme. Therefore, \cref{thm:1} also holds for the classical Grassmannians.
\end{remark}

A parallel result holds for Hochschild homology 

 \begin{theorem}
  \label{thm:hos}
  Suppose $\mathcal{E}$ is a $G$-smooth complex of rank $r \ge 0$ over a smooth projective variety $X$. Then there is an isomorphism of bi-graded $\mathrm{Cl}(r)$-representations
  \begin{equation*}
    \bigoplus_{m\in \mathbb{Z}} \mathrm{HH}_*(\mathrm{Gr}_X(\mathcal{E},m))\cong \bigoplus_{m\in \mathbb{Z}} \mathrm{HH}_*(\mathrm{Gr}_X(\mathcal{E}^{\vee}[1],m))\otimes \mathrm{F}(r)
  \end{equation*}
  where $\mathrm{F}(r)$ is equipped with dual length grading and is concentrated in Hochschild degree zero. In particular, the operators $e,f, p_{i}, q_{i}$ preserve the homological grading. 
\end{theorem}

\begin{proof}
For the construction of Hochschild homology, we refer to \cite{caldararu2003mukai}, which associates to each smooth projective complex variety $Y$ a graded vector space $\mathrm{HH}_*(Y)$, and to any Fourier--Mukai transform $\Phi \colon \mathrm{D}^\mathrm{b}_{\mathrm{coh}}(Y_1) \to \mathrm{D}^\mathrm{b}_{\mathrm{coh}}(Y_2)$ an induced map
\[
\Phi_{\mathrm{HH}}  \colon \mathrm{HH}_*(Y_1) \to \mathrm{HH}_*(Y_2),
\]
which preserves the homological degree and is functorial with respect to composition and identities; see \cite[\S 5]{caldararu2003mukai} for details. The remaining proof is the same as \cref{thm:1}, while the fact that the $\mathrm{Cl}(r)$ action preserves the homological grading follows from the fact that $\Phi_{\mathrm{HH}}$ preserves the homological degree.
\end{proof}

\begin{remark}
A similar result holds for cohomology. Specifically, the compatibility of the Mukai vector with Fourier--Mukai transforms (see \cite[Cor.~5.29, Lem.~5.32]{huybrechts2006fm}) implies that the induced maps on cohomology satisfy the same relations as those on $K$-theory. Thus, we obtain an isomorphism of $\mathrm{Cl}(r)$-representations:
\begin{equation*}
  \bigoplus_{m\in \mathbb{Z}} \mathrm{H}^*(\mathrm{Gr}_X(\mathcal{E},m), \mathbb{Q})\cong \bigoplus_{m\in \mathbb{Z}} \mathrm{H}^*(\mathrm{Gr}_X(\mathcal{E}^{\vee}[1],m), \mathbb{Q})\otimes \mathrm{F}(r).
\end{equation*}
However, this isomorphism cannot be lifted to an isomorphism of graded representations for the cohomological grading, as the correspondences are not homogeneous.
 \end{remark}

\section{Perverse stable sheaves on blow-ups, the extended algebra \texorpdfstring{$\mc{E}(r)$}{} and the Morita equivalence}
\label{sec:666}
Now we recall $\mathrm{M}^0(l,n)$ from \cref{sec:gra}. For any $n,l\in \mathbb{Z}$, by \cref{thm:NY1}, we have isomorphisms
\begin{equation*}
    \mathrm{M}^0(l,n)\cong \mathrm{Gr}_{\mc{M}_H(n)}(\mc{U}_o,-l), \quad \mathrm{M}^0(l,n+l)\cong \mathrm{Gr}_{\mc{M}_H(n)}(\mc{U}_o^{\vee}[1],-l).
\end{equation*}
By \cref{thm:1} and \cref{prop:4}, for any $n\in \mathbb{Z}$, we have a $\mathrm{Cl}(r)$ action on $\oplus_{n\in \mathbb{Z}} \mb{H}^*( \mathrm{M}^0(l,n))$ and maps 
\begin{align*}
  e:\mb{H}^*(\mathrm{M}^0(l+r,n+l+r))\to \mb{H}^*(\mathrm{M}^0(l,n)), \\ f:\mb{H}^*(\mathrm{M}^0(l,n))\to \mb{H}^*(\mathrm{M}^0(l+r,n+l+r))
\end{align*}
such that \cref{eq3} holds. Thus $e,f$ induce endomorphisms on $\bigoplus_{l,n\in \mb{Z}} \mb{H}^*(\mathrm{M}^0(l,n))$. Let $\mathcal{E}(r)$ be the extended algebra over $\mathrm{Cl}(r)$ with two additional elements $e$ and $f$ with the relations \eqref{eq3}
\begin{equation*}
     p_{i}e=0, \quad fq_{i}=0, \quad fe=1, \quad ef=p_{[r]}q_{[r]}.
\end{equation*}
Then by the above explanation, the algebra $\mc{E}(r)$ naturally acts on
\begin{equation}
\label{eq:act}
    \mc{E}(r) \curvearrowright	\bigoplus_{l,n\in \mb{Z}}\mb{H}^*(\mathrm{M}^0(l,n)).
\end{equation}

On the other hand, we can also obtain $\mc{E}(r)$ through the Morita theory: let the infinite-dimensional Clifford algebra $\mf{E}(r)$ be the unital algebra generated by elements $P_{a,i}, Q_{a,i}, a\in \mb{Z}, i\in [r]$ and an invertible element $E$ with the relations 
\begin{equation}
  \label{eq:clifford}
  \begin{split}
    P_{a,i}^2=0, \quad Q_{a,i}^2=0, \quad EP_{a,i}=P_{a-1,i}E, \quad EQ_{a,i}=Q_{a-1,i}E \quad  \text{ for all } a,i\\
  \{P_{a,i}, P_{b,j}\}=0, \quad \{Q_{a,i}, Q_{b,j}\}=0, \quad \{P_{a,i}, Q_{b,j}\}=\delta_{ij}\delta_{ab} \quad \text{for all $a,b,i,j$,}
    \end{split}
\end{equation}
Let $F:=E^{-1}$. In this section, we will prove that

 \begin{proposition}
  \label{prop:bimodule}
  Let $\Lambda_{r}$ be the $\mf{E}(r)$-left module with a generator $1_{\Lambda_r}$ with the relation 
\begin{equation*}
P_{a,i}|1_{\Lambda_r} =0, \quad \forall a>0, i\in [r].
\end{equation*}
  Then we have $\mc{E}(r)\cong \mathrm{End}_{\mf{E}(r)}(\Lambda_{r})^{op}$. Moreover, let $\mathrm{Mod}_{\infty}^R$ be the abelian category of right-admissible left $\mf{E}(r)$-modules (which will be defined in \cref{sec:Morita}). Then $\Lambda_{r}$ is a small projective generator of $\mathrm{Mod}_{\infty}^R$.

  As a right $\mc{E}(r)$-module,  $\Lambda_r$ is generated by $F^{l}1_{\Lambda_r}$  with the relation $F^{l}1_{\Lambda_r}-F^{l+1}1_{\Lambda_r}e=0$ for all $l\in \mathbb{Z}$.
\end{proposition}
Let $\mathrm{Mod}_0$ (resp. $\mathrm{Mod}_{\infty}$) be the abelian category of left $\mc{E}(r)$-modules (resp. left $\mf{E}(r)$-modules). By the Morita theory (see Theorem 2.5 of \cite{Morita}), there is a Morita equivalence between $\mathrm{Mod}_0$ and $\mathrm{Mod}_{\infty}^R$ by the functors: 
\begin{align*}
  \mf{H}_{\infty}: \mathrm{Mod}_0 &\to \mathrm{Mod}_{\infty}, \quad V\mapsto \Lambda_{r}\otimes_{\mc{E}(r)}V, \\
  \mc{H}_0: \mathrm{Mod}_{\infty} &\to \mathrm{Mod}_0, \quad W\mapsto \mathrm{Hom}_{\mf{C}(r)}(\Lambda_{r},W).
\end{align*}

\subsection{The bimodule structure}\label{sec:Morita} 
We first give the definition of admissible left $\mf{E}(r)$-modules.
For any $I=(d_0<d_1<\cdots<d_k)\subset [r]$ and $a,b\in \mb{Z}$, let
$$P_{a,I}:=P_{a,d_{k}}P_{a,d_{k-1}}\cdots P_{a,d_0},Q_{b,I}:=Q_{b,d_{0}}Q_{b,d_{1}}\cdots Q_{b,d_k}.$$ For $V\in \mathrm{Mod}_\infty$,  let
\begin{gather*}
    \mc{H}_m(V):=\bigcap_{a> m, i\in [r]} \ker(P_{a,i}), \quad  \mc{K}_m(V):=\bigcap_{a\leq m, i\in [r]} \ker(Q_{a,i}). \\
    \quad \mc{H}_{\infty}(V):=\bigcup_{m\in \mb{Z}}\mc{H}_m(V),\quad 
          \quad \mc{K}_{-\infty}(V):=\bigcup_{m\in \mb{Z}}\mc{K}_m(V).
  \end{gather*}
We notice that $\mc{H}_{\infty}(V)$ and $\mc{K}_{-\infty}(V)$ are also $\mf{C}(r)$ sub-representations of $V$. The representation $V$ is defined to be \textbf{right admissible} (resp. \textbf{left admissible}, \textbf{admissible}) if $\mc{H}_{\infty}(V)=V$ (resp. $\mc{K}_{-\infty}(V)=V$, $\mc{H}_{\infty}(V)\cap\mc{K}_{-\infty}(V)=V$). Let $\mathrm{Mod}_{\infty}^{R}$ be the full subcategory of $\mathrm{Mod}_{\infty}$ consisting of right admissible representations.

Before proving \cref{prop:bimodule}, we give the following corollary, which describes the functors $\mc{H}_0$ and $\mf{H}_{\infty}$ explicitly.
  \begin{corollary}
    \label{cor:equivalence1} For any left $\mf{E}(r)$-module $W$, we have 
  \begin{equation*}
  \mc{H}_0(W)\cong \{w\in W \mid P_{a,i}w=0, \forall a>0, i\in [r]\};  
  \end{equation*}
  and for any left $\mc{E}(r)$-module $V$, we have the \textbf{colimit}
  \begin{equation}
    \label{colimit}
  \mf{H}_{\infty}(V)=\varinjlim_{a\to +\infty} (t^{a}V)  
  \end{equation}
  where $t^aV$ is a copy of $V$ for every $a\in \mathbb{Z}$ and the morphism from $t^{a}V$ to $t^{a+1}V$ is given by $t^av\to t^{a+1}\cdot (ev)$. In particular, $t^aV$ is a $\mb{Z}$-submodule of $\mf{H}_{\infty}(V)$.

\end{corollary}
\begin{proof}
  The formula for $\mc{H}_0(W)$ follows from the definition of $\Lambda_r$. The formula \cref{colimit} follows from the fact that, as a right $\mc{E}(r)$-module, $\Lambda_r$ can be written as
    \begin{equation}
      \label{lambda:r}
      \Lambda_{\infty,r}:= \bigoplus_{a\in \mb{Z}} t^{a}\mc{E}(r)/\langle t^{a} - t^{a+1}e, a\in \mb{Z}\rangle\cong\varinjlim_{a\to +\infty} (t^{a}\mc{E}(r))  
    \end{equation}
  where $t^{a}\mc{E}(r)$ is a copy of $\mc{E}(r)$ for every $a\in \mathbb{Z}$ and the morphism from $t^{a}\mc{E}(r)$ to $t^{a+1}\mc{E}(r)$ is given by $t^av\to t^{a+1}\cdot (e\cdot v)$ for every $v\in \mc{E}(r)$. 
\end{proof}

The key to proving \cref{prop:bimodule} is the identification of $\Lambda_r$ and $\Lambda_{\infty,r}$ as an $\mf{E}(r)$-$\mc{E}(r)$ bimodule.
\begin{lemma}
  \label{lem:tilde}
    Let $\widetilde{\Lambda}_{r}$ be the $\mf{E}(r)-\mc{E}(r)$-bimodule with a generator $1_{\widetilde{\Lambda}_r}$ and the relation 
\begin{align}
  \label{eq:verma}
  P_{a,i}1_{\widetilde{\Lambda}_r}=0, \quad \forall a>0, i\in [r], \\
 \nonumber 1_{\widetilde{\Lambda}_r}e = E 1_{\widetilde{\Lambda}_r}, \quad 
  1_{\widetilde{\Lambda}_r}f = FP_{0,[r]}Q_{0,[r]} 1_{\widetilde{\Lambda}_r}, \\
 \nonumber 1_{\widetilde{\Lambda}_r}p_{i} = P_{0,i} 1_{\widetilde{\Lambda}_r}, \quad
  1_{\widetilde{\Lambda}_r}q_{i} = Q_{0,i} 1_{\widetilde{\Lambda}_r}.
\end{align}
Then $\widetilde{\Lambda}_{r}\cong \Lambda_r$ as a left $\mf{E}(r)$-module and $\widetilde{\Lambda}_{r}\cong \Lambda_{\infty,r}$ as a right $\mc{E}(r)$-module. 
\end{lemma}
\begin{proof}
    As a left $\mf{E}(r)$-module, $\widetilde{\Lambda}_r$ is generated by $1_{\widetilde{\Lambda}_r}$ and hence is a quotient of $\Lambda_r$. To prove the quotient map is an isomorphism, we only need to prove that there is a right $\mc{E}(r)$-structure on $\Lambda_{r}$ such that \cref{eq:verma} also holds if we replace $1_{\widetilde{\Lambda}_{r}}$ by $1_{\Lambda_r}$. Let $\mf{J}(r)$ be the left ideal of $\mf{E}(r)$ generated by $P_{a,i}$ for all $a>0$ and $i\in [r]$. Then we have
$\Lambda_{r}\cong\mf{E}(r)/\mf{J}(r)$. Moreover, by direct computation, for $x\in \mf{J}(r)$, we have
  \begin{equation}
    \label{eq:revise111}
  xP_{0,i}\in \mf{J}(r), \quad xQ_{0,i}\in \mf{J}(r), \quad xE\in \mf{J}(r), \quad xFP_{0,[r]}Q_{0,[r]}\in \mf{J}(r).
  \end{equation}
  In particular, the operators given by right multiplication,
  \begin{equation*}
    e:=E, \quad f:=FP_{0,[r]}Q_{0,[r]}, \quad p_{i}:=P_{0,i}, \quad q_{i}:=Q_{0,i}
  \end{equation*}
  define a right $\mc{E}(r)$-module structure on $\Lambda_{r}$.
    
  As a right $\mc{E}(r)$-module, $\widetilde{\Lambda}_r$ is generated by $F^{a}1_{\widetilde{\Lambda}_r}$ for all $a\in \mb{Z}$, as every monomial element in $\mf{E}(r)$ can be written in the form 
  \begin{equation*}
F^aP_{b_1,I_1}Q_{b_1,J_1}P_{b_2,I_2}Q_{b_2,J_2}\cdots P_{b_k,I_k}Q_{b_k,J_k} 
  \end{equation*}
with $b_1 < b_2 < \cdots < b_k$ and $I_j, J_j \subset [r]$ for $1 \leq j \leq k$.
    Hence $\widetilde{\Lambda}_r$ is a quotient of $\Lambda_{\infty,r}$ as right $\mc{E}(r)$-modules. To prove that the quotient map is an isomorphism, we only need to prove that there exists a $\mf{E}(r)$-left module structure on $\Lambda_{\infty,r}$ such that $Ft^a=t^{a+1}$ for all $a\in \mb{Z}$ and 
    \begin{align}
      \label{eq:revise}
  P_{a,i}t^0=0 \text{ for all } a>0, i\in [r], \quad
 t^0p_{i} = P_{0,i} t^0, \quad
  t^0q_{i} = Q_{0,i} t^0.
\end{align}

To construct the $\mf{E}(r)$-left module structure on $\Lambda_{\infty,r}$, we first inductively define $p_{a,i},q_{a,i}\in \mc{E}(r)$ for all $a\leq 0$ by $p_{0,i}=p_{i}$ and $q_{0,i}=q_{i}$ and
    \begin{equation}
    \label{def:ind}
            p_{a,i}:=\sum_{I\subset [r]}(-1)^{\mathrm{len}(I)}e_{I}p_{a+1,i}f_I,\quad  q_{a,i}:=\sum_{I\subset [r]}(-1)^{\mathrm{len}(I)}e_Iq_{a+1,i} f_I.
        \end{equation}
        for $a<0$, where $e_I$ and $f_I$ for all $I\subset [r]$ are defined by \cref{eq:inverse}. By direct computation, for all $a \leq 0$, we have
        \begin{equation}
           \label{eq222}
    ep_{a,i}=p_{a-1,i}e, \quad fp_{a-1,i}=p_{a,i}f, \quad eq_{a,i}=q_{a-1,i}e, \quad  fq_{a-1,i}=q_{a,i} f.
        \end{equation}
        Now we prove that for all $a,b\leq 0$ we have
      \begin{gather}
    \label{eq:er4}
    p_{a,i}^2 = q_{a,i}^2 = 0, \quad
    \{p_{a,i}, p_{b,j}\}=\{q_{a,i}, q_{b,j}\}=0, \quad \{p_{a,i}, q_{b,j}\}=\delta_{i,j}\delta_{a,b}.
      \end{gather}
       By \cref{cor:faith} and \cref{def:ind}, we can reduce to the case that $a=0$. If $b=0$, \cref{eq:er4} just follows from the definition of $\mc{E}(r)$.
    If $b<0$, for any $J\subset [r]$, we notice that 
    \begin{equation*}
        p_{0,i}e_{J}=\begin{cases}
            0 & i\not\in J, \\
            (-1)^{|\{j\in J|j<i\}|} e_{J-\{i\}} & i\in J,
        \end{cases}
        \quad q_{0,i}e_{J}=\begin{cases}
            0 & i\in J, \\
            (-1)^{|\{j\in J|j<i\}|} e_{J+\{i\}}& i\not\in J.   
        \end{cases}
    \end{equation*}
   We have $\{p_{0,i},p_{b,j}\}e_Jf_J=0$ for all $J\subset [r]$, as if $i\not\in J$ we have
    \begin{align*}
        p_{0,i}p_{b,j}e_Jf_J=(-1)^{\mathrm{len}(J)}p_{0,i}e_{J}p_{b+1,j}f_{J}=0, \quad 
        p_{b,j}p_{0,i}e_{J}f_{J}=0,
    \end{align*}
    and if $i\in J$ we have 
    \begin{align*}
        p_{0,i}p_{b,j}e_{J}f_{J}=(-1)^{|\{j\in J|j\geq i\}|}e_{J-i}p_{b+1,j}f_{J}, \\
        p_{b,j}p_{0,i}e_{J}f_{J}=(-1)^{|\{j\in J|j> i\}|}e_{J-i}p_{b+1,j}f_{J}.
    \end{align*}
    As $1=\sum_{J\subset [r]} e_Jf_J$, we have $\{p_{0,i},p_{b,j}\}=0$. Other formulas follow from a similar computation.

    Now we come back to $\Lambda_{\infty,r}$. For any $a\in \mb{Z}$, $t^a\mc{E}(r)$ in \cref{lambda:r} is a $\mc{E}(r)$-submodule of $\Lambda_{\infty,r}$, as the map from $t^{a}\mc{E}(r)$ to $t^{a+1}\mc{E}(r)$ is injective. Moreover, by the colimit property, for any integer $a\in \mathbb{Z}$,
$\Lambda_{\infty,r}$ is the union of all $t^m\mc{E}(r)$ for $m\geq a$. The operators $E,F$ on
    $\bigoplus_{a\in \mb{Z}} t^{a}\mc{E}(r)$ defined by $E(t^m\cdot v)=t^{m-1}\cdot v$ and $F(t^m\cdot v)=t^{m+1}\cdot v$ for any $m\in \mb{Z}$ and $v\in \mc{E}(r)$ descend to operators on $\Lambda_{\infty,r}$. 
    By \cref{eq222}, for all $m \geq a$, the operators $P_{a,i}$ and $Q_{a,i}$ on $t^m \mc{E}(r)$, defined by
\begin{equation}
  \label{def:plqi}
  P_{a,i}(t^m\cdot v):= t^m\cdot p_{a-m,i}v, \quad Q_{a,i}(t^m\cdot v):=t^m\cdot q_{a-m,i}v,
\end{equation}
descend to well-defined operators on $\Lambda_{\infty,r}$. Moreover, by \cref{eq222} and \cref{eq:er4}, the operators $P_{a,i},Q_{a,i},E,F$ form a left $\mf{C}(r)$-module structure on $\Lambda_{\infty,r}$ with the generator $t^0$. By direct computation, \cref{eq:revise} also holds. Thus, we finish the proof.
\end{proof}

\begin{proof}[Proof of \cref{prop:bimodule}]The right $\mc{E}(r)$-structure on $\Lambda_r$ follows from \cref{lem:tilde}. Now we prove that $\mc{E}(r)\cong \mathrm{End}_{\mf{E}(r)}(\Lambda_{r})^{op}$. We apply the equivalence $\Lambda_r\cong \Lambda_{\infty,r}$ and notice that   
\begin{equation}
\label{ineqeee}
  \mathrm{Hom}_{\mf{E}(r)}(\Lambda_{r},\Lambda_{r})\cong \{x\in \Lambda_{\infty,r}|
P_{a,i}x=0 \text{ for all } a>0, i\in [r]\}\end{equation}
which contains $t^0\mc{E}(r)$ by \cref{eq:revise}. On the other hand, let $w=t^m v$ in \cref{ineqeee} for $m>0$ and $v\in \mc{E}(r)$. Then by \cref{cor:faith}, we have $v=P_{m,[r]}Q_{m,[r]}v$ and thus
    \begin{equation*}
        w=P_{m,[r]}Q_{m,[r]}(t^m\cdot v)=t^m \cdot (p_{[r]}q_{[r]}v)=t^m\cdot (efv)=t^{m-1}\cdot (fv),
    \end{equation*}
    and hence is in the image of $t^{m-1}\mc{E}(r)$. Thus by induction $w\in t^{0}\mc{E}(r)$.

Finally, we prove that $\Lambda_r$ is a small projective generator of $\mathrm{Mod}_{\infty}^{R}$. First, by definition, $\Lambda_r$ is a right admissible left $\mf{E}(r)$-module. 
Moreover, since $\Lambda_r$ is finitely presented as a left $\mathfrak{E}(r)$-module, it is a small (i.e., compact) object in $\mathrm{Mod}_{\infty}$, and thus also in $\mathrm{Mod}_{\infty}^R$. Second, we prove that $\Lambda_r$ is projective in $\mathrm{Mod}_{\infty}^{R}$. We only need to prove that given right admissible left $\mf{E}(r)$-modules $W_{1}\subset W_2$, for any $w\in W_2$ such that $P_{a,i}w\in W_1$ for all $a>0, i\in [r]$, there exists $v\in W_1$ such that $P_{a,i}(w-v)=0$ for all $a>0, i\in [r]$. Taking $a_0>0$ such that $P_{a,i}w=0$ for all $a\geq a_0$, we define
    \begin{equation*}
        v:=w-P_{a_0-1,[r]}Q_{a_0-1,[r]}P_{a_0-2,[r]}Q_{a_0-2,[r]}\cdots P_{1,[r]}Q_{1,[r]}w.
    \end{equation*}
    Then $v\in W_1$ and $P_{a,i}(w-v)=0$ for all $a>0, i\in [r]$ by induction and \cref{cor:faith}.
    
    We still need to prove that $\Lambda_r$ is a generator of $\mathrm{Mod}_{\infty}^{R}$. For any $W\in \mathrm{Mod}_{\infty}^{R}$ and any $w\in W$, there exists $a_0> 0$ such that $P_{a,i}w=0$ for all $a\geq a_0, i\in [r]$. Then we have a morphism $\phi:\Lambda_r\to W$ by sending $1_{\Lambda_r}$ to $E^{a_0-1}w$ such that $w$ is in the image of $\phi$. Hence $W$ is a quotient of a direct sum of copies of $\Lambda_r$.
\end{proof}

\subsection{A precise formulation of \texorpdfstring{\cref{eq:act}}{}} Now we give a precise formulation of \cref{eq:act} and its cohomological variants. In particular, we formulate the degree shift of the action under different gradings, which can be obtained through direct computations.
\begin{proposition}
\label{shift1}
    There exists a $\mc{E}(r)$-representation on $$ \bigoplus_{l,n\in \mb{Z}} \mathrm{G}_0(\mathrm{M}^0(l,n))$$
with two gradings induced by $l$ and $n$, which we call the first and second grading. The operators $e,f,p_i,q_i$ shift the first degree and second degree by
\begin{equation}
    \label{deg:shift}
    -r,r,-1,1 \text{ and } -l,l+r,0,0
\end{equation}
respectively. Moreover, if \cref{assumptionS} is satisfied, there exists a $\mc{E}(r)$ action on 
\begin{equation*}
    \bigoplus_{l,n\in \mb{Z}} \mathrm{HH}_*(\mathrm{M}^0(l,n))
\end{equation*}
which is triply graded by the first and second gradings and the homological grading. The degree shifts in the first and second gradings are the same as in \cref{deg:shift}, and the homological grading is preserved under the action.
\end{proposition}
\begin{proposition} 
\label{shift2}
Still assuming \cref{assumptionS}, let 
  \begin{equation*}
    \mb{CH}^m(\mathrm{M}^0(l,n)):=\mathrm{CH}^{m-\frac{1}{2}(l(l+1))}(\mathrm{M}^0(l,n)).
    \end{equation*}
    Then there exists a $\mc{E}(r)$-action on 
    \begin{equation*}
        \bigoplus_{l,n}\mb{CH}^*(\mathrm{M}^0(l,n))
    \end{equation*}
    which is triply graded by the first grading, second grading, and cohomological grading. The degree shifts in the first and second gradings are the same as in \cref{deg:shift}, and the operators $e,f,p_i,q_i$ shift the cohomology degree by
    \begin{equation*}
        -rl+\frac{1}{2}r(r-1), r(l+r)-\frac{1}{2}r(r-1), i, -i.
    \end{equation*}
    Similar arguments also work for Hodge cohomology groups if we replace the grading with the bigrading.
\end{proposition}
\begin{proof}
    We remark that, unlike in \cref{shift1}, here we consider the dual representation in \cref{rem:dualrep}. The degree shift follows from \cref{lem:grading2}.
\end{proof}

\section{The admissible representation of \texorpdfstring{$\mf{C}(r)$}{} and the colimit}
\label{sec:fermion}
In this section, we study the image of $\bigoplus_{l,n}\mb{H}^*(\mathrm{M}^0(l,n))$ under the functor $\mf{H}_{\infty}$.  In particular, we need to prove the following two isomorphisms
\begin{align}
  \label{stable:limit}
    \mf{H}_{\infty}(\bigoplus_{l,n}\mb{H}^*(\mathrm{M}^0(l,n)))\cong \bigoplus \mb{H}^*(\mc{M}_{H_{\infty}}(l,n)), \\
    \nonumber \mf{H}_{\infty}(\bigoplus_{l,n}\mb{H}^*(\mathrm{M}^0(l,n)))\cong  \bigoplus_{n\in \mb{Z}}\mb{H}^*(\mc{M}_{H}(n))\otimes \mf{F}(r),
\end{align}
as $\mathbb{Z}$-modules and $\mf{E}(r)$-representations respectively, where we define the Fermionic Fock space $\mf{F}(r)$ in a way compatible with our geometric correspondences:
\begin{definition}
\label{def:maya}
  An $r$-Maya diagram $\mf{I}$ is defined as a sequence of subsets of $[r]$
\begin{equation*}
 ( \cdots I_{m-1}, I_{m}, I_{m+1},\cdots )\in (2^{[r]})^{\mb{Z}}
\end{equation*}
where $I_{m}\subset [r]$ for all $m\in \mb{Z}$ such that $I_{m}=[r]$ for  $m\ll 0$, and  $I_{m}=\emptyset$ for $m\gg 0$. The \textbf{Fermionic Fock space} $\mf{F}(r)$ is defined as the \textbf{free} $\mathbb{Z}$-module generated by formal variables $v_{\mf{I}}$ for all  $r$-Maya diagrams $\mf{I}$. 
For any $r$-Maya diagram $\mf{I}$, we formally write $v_{\mf{I}}$ as the infinite tensor product  $v_\mf{I}:=\cdots \otimes v_{I_{-1}}\otimes v_{I_{0}}\otimes v_{I_{1}}\otimes \cdots$.
\end{definition}

\subsection{The admissible representations} 
The $\mathrm{Cl}(r)$-representation on $\mathrm{F}(r)$ in \cref{eq:rep1} induces a natural $\mf{E}(r)$-representation on $\mf{F}(r)$ where $Ev_{\mf{I}}:=v_{\mf{I}^-}$, $Fv_{\mf{I}}:=v_{\mf{I}^+}$
such that the $i$-th term of $\mf{I}^{\pm}$ is defined as the $(i\mp 1)$-th term of $\mf{I}$ and 
\begin{align*}
  P_{a,i}v_{\mf{I}}:= (-1)^{\sum_{k>a}\mathrm{len}(I_k)}\cdots \otimes v_{I_{a-1}}\otimes (p_{i}v_{I_{a}}) \otimes v_{I_{a+1}} \cdots, \\
  Q_{a,i}v_{\mf{I}}:= (-1)^{\sum_{k>a}\mathrm{len}(I_k)}\cdots \otimes v_{I_{a-1}}\otimes (q_{i}v_{I_{a}}) \otimes v_{I_{a+1}} \cdots.
\end{align*} 
 Moreover, given an automorphism $T$ on a $\mathbb{Z}$-module $W$,  $\mf{F}(W,T) := W\otimes \mf{F}(r)$ has a canonical $\mathfrak{C}(r)$-action by
\begin{align*}
  E(w\otimes v_{\mf{I}}):=T(w)\otimes v_{\mf{I}^{-}}, \quad F(w\otimes v_{\mf{I}}):=T^{-1}(w)\otimes v_{\mf{I}^+}, \\
  P_{l,i}(w\otimes v_{\mf{I}}):=w\otimes P_{l,i}v_{\mf{I}}, \quad Q_{l,i}(w\otimes v_{\mf{I}}):=w\otimes Q_{l,i}v_{\mf{I}}.
\end{align*}
 We prove that any admissible $\mf{C}(r)$-representation is of the form $\mf{F}(W,T)$.
\begin{proposition}
  \label{thm:equivalence}Given $V\in \mathrm{Mod}_{\infty}$, then the operators 
  \begin{equation*}
     E_{[r]}:=Q_{0,[r]}E \text{ and } F_{[r]}:=P_{1,[r]}F=FP_{0,[r]}
  \end{equation*}
are both closed on $W:=\mc{H}_{0}(V)\cap \mc{K}_{0}(V)$, and their restrictions to $W$ are inverses of each other. Moreover, let $T:=E_{[r]}|_W$; there is a canonical isomorphism of $\mf{C}(r)$-representations $\mf{F}(W,T)\cong \mc{K}_{-\infty}(V)\cap\mc{H}_{\infty}(V)$ given by
      \begin{align}
               \label{admissible}
          w\otimes v_\mf{I}\to  \lambda_{\mf{I}}\cdots Q_{2,I_2}Q_{1,I_1}P_{0,[r]-I_{0}} P_{-1,[r]-I_{-1}}\cdots w
      \end{align}
where $\lambda_{\mf{I}}:=(-1)^{(\sum_{k=-\infty}^0kr(r-\mathrm{len}(I_k)))}$ is a sign in $\{\pm 1\}$.
\end{proposition}
\begin{proof}
 For any $w\in \mc{K}_0(V)$, $a\leq 0$ and $i\in [r]$, we have 
  \begin{align*}
     &Q_{a,i}Q_{0,[r]}Ew=
    \begin{cases}
     0 & a=0, \\
(-1)^rQ_{0,[r]}EQ_{a+1,i}w=0 &a<0, 
    \end{cases}
    \\ &  Q_{a,[i]}P_{1,[r]}Fw=(-1)^rP_{1,[r]}FQ_{a-1,i}w=0.
  \end{align*}
  and hence $\mc{K}_0(V)$ is closed under $E_{[r]}=Q_{0,[r]}E$ and $F_{[r]}=P_{1,[r]}F$. Similarly, $\mc{H}_0(V)$ is also closed under  $E_{[r]}$ and $F_{[r]}$. Moreover, by \cref{lem:45} we have
  \begin{equation*}
    (Q_{0,[r]}P_{0,[r]})|_W=\operatorname{Id}_W=(P_{1,[r]}Q_{1,[r]})|_W
  \end{equation*}
 and hence $E_{[r]}$ and $F_{[r]}$ are inverses of each other on $W$.

    By \cref{eq:clifcong}, for any integer $a$, the homomorphisms
  \begin{gather*}
    \mc{H}_{a-1}(V)\otimes \mathrm{F}(r)\to \mc{H}_a(V), \quad w\otimes v_{I}\to Q_{a,I}w, \\
    \mc{K}_{a}(V)\otimes \mathrm{F}(r)\to \mc{K}_{a-1}(V), \quad w\otimes v_{I}\to P_{a,[r]-I}w
  \end{gather*}
  are isomorphisms of $\mathrm{Cl}(r)$-representations, where the action on the right sides is given by the action of $P_{a,i}$ and $Q_{a,i}$. Thus by induction, \cref{admissible} is an isomorphism of $\mb{Z}$-modules and the operators $P_{a,i},Q_{b,j}$ for all $i,j\in [r],a,b\in\mb{Z}$ are compatible with \cref{admissible}. We only need to check that the action of the shifting operators $E$ and $F$ is compatible with \cref{admissible}. It follows from direct computation and that over $W$ we have $E|_{W}=(P_{0,[r]}T)|_{W}$ and $F|_W=(Q_{1,[r]}T^{-1})|_W$.
  \end{proof}

\begin{remark}
    Like \cref{prop:bimodule}, one can also show an equivalence between the abelian category of $\mathbb{Z}[t,t^{-1}]$-modules and the abelian category of admissible $\mf{C}(r)$-representations. The point is that the module $\mf{F}(r)[t,t^{-1}]$ is the module generated by a vacuum vector $|0\rangle$ with the relation $P_{a,i}|0\rangle=0$ for all $a>0,i\in [r]$ and $Q_{b,i}|0\rangle=0$ for all $b\leq 0,i\in [r]$. Moreover, $\mf{F}(r)[t,t^{-1}]$  is a small projective generator of the abelian category of admissible $\mf{C}(r)$-representations.
\end{remark}

 \subsection{Bounded bigraded \texorpdfstring{$\mc{E}(r)$}{}-representations and the parabolic colimit}
 \label{sec:boundedrep}
 For a $r$-Maya diagram $\mf{I}$, we define the \textbf{central charge} and the \textbf{weight} as 
\begin{align*}
  c_{\mf{I}}:=\sum_{l\in \mb{Z}_{>0}}\mathrm{len}(I_l)-\sum_{l\in \mb{Z}_{\leq 0}}(r-\mathrm{len}(I_l)), \\ 
  w_{r,\mf{I}}:=\sum_{l\in \mb{Z}_{>0}}l\cdot \mathrm{len}(I_l)-\sum_{l\in \mb{Z}_{\leq 0}}l(r-\mathrm{len}(I_l)).
\end{align*}
They induce a bigrading on $\mf{F}(r)$ by
\begin{equation*}
  \mf{F}(r)=\bigoplus_{l,m\in \mb{Z}}\mf{F}(r)_{l,m}, \quad \mf{F}(r)_{l,m}=\langle v_{\mf{I}}|c_\mf{I}=l, w_{r,\mf{I}}=m\rangle.
\end{equation*}

 Let $\mf{F}(r)^{\leq 0}:= \mc{H}_0(\mf{F}(r))$, which is the free $\mb{Z}$-module generated by all Maya diagrams $\mf{I}$ such that $I_{m}=\emptyset$ for all $m>0$. There is a canonical bigraded $\mc{E}(r)$-action on $\mf{F}(r)^{\le 0}$ induced from the functor $\mc{H}_0$. By direct computation, the operators $p_{0,i},q_{0,i},e,f, e_{[r]},f_{[r]}$ shift the central charge degree and the weight degree on $\mf{F}(r)^{\leq 0}_{l,\bullet} = \bigoplus_{m \in \mathbb{Z}} \mf{F}(r)_{l,m}^{\le 0}$, respectively, by
  \begin{gather}
    \label{lem:grading}  
    -1,1,-r,r,0,0, \text{ and }  0,0,-l,l+r,l,-l,
  \end{gather} 
which is similar to the degree shift in \cref{shift1} and \cref{shift2}. Moreover, for any $l>0$, we have $\mathrm{G}_0(\mathrm{M}^0(l,n))=0$ for all $n\in \mb{Z}$ by \cref{thm:NY1} and $\mf{F}(r)_{l,n}^{\leq 0}=0$ by the definition of central charge. Given $n\ll 0$, by the Bogomolov--Gieseker inequality (and also \cref{thm:NY1}) and the definition of the weight, we have $\mathrm{G}_0(\mathrm{M}^0(l,n))=0$ and $\mf{F}(r)_{l,n}^{\leq 0}=0$ for all $l\in \mb{Z}$. We summarize the above similarities to give the definition of a bounded bigraded $\mc{E}(r)$-representation:
\begin{definition}
\label{def:bounded.bigraded}
Given $V\in \mathrm{Mod}_0$, we say that $V$ is \textbf{bounded and bigraded} if there exists a bigrading by two degrees, which we call the central charge and weight degrees, respectively,
  $$V\cong \bigoplus_{l,n\in \mb{Z}}V_{l,n},$$
   such that there exists $n_0\in \mb{Z}$ such that $V_{l,n}=0$ holds for all $l>0$ or $n<n_0$, and on $V_{l,n}$ the operators $p_{0,i},q_{0,i},e,f$ shift the central charge and the weight degree by the formula \cref{lem:grading}. If $V$ is bounded and bigraded, the bigrading on $t^mV$ for all $m\in \mb{Z}$ by 
    $t^mV_{l,n}:=V_{l-mr,n-ml+\frac{m(m-1)}{2}r}$ induces a bigrading on $\mf{H}_{\infty}(V)$ by \cref{cor:equivalence1}.
\end{definition}

Given a bounded bigraded $\mc{E}(r)$-representation $V$, the operator $E,F,P_{0,i},Q_{0,i}$ shift the central charge degree and weight degree by $-r,r,-1,1$ and $-l,l+r,0,0$ on $\mf{H}_{\infty}(V)_{l,n}$ respectively. 
\begin{proposition}
  \label{prop:big}
  For any bounded and bigraded $\mc{E}(r)$-representation $V$, 
  we have isomorphisms of bigraded $\mc{E}(r)$ and $\mf{E}(r)$ representations 
    \begin{equation}
    \label{eq:standard}
    V\cong V_{0,\bullet}\otimes \mf{F}(r)^{\leq 0}, \quad \mf{H}_{\infty}(V)\cong V_{0,\bullet}\otimes \mf{F}(r),
  \end{equation}
  where $V_{0,\bullet}:=\bigoplus_{n\in \mb{Z}}V_{0,n}$. Moreover, for any $l,n\in \mb{Z}$, we have the parabolic property
    \begin{equation}
    \label{eq:74}
    \mf{H}_{\infty}(V)_{l,n}\cong V_{l-rm,n-ml+\frac{1}{2}(m-1)mr} \text{ if } m\gg 0.
  \end{equation}
\end{proposition}
\begin{proof}  By \cref{thm:equivalence}, the key point is proving that $\mf{H}_{\infty}(V)$ is admissible and 
    \begin{equation}
    \label{eq:73} 
    (\mc{H}_0\cap \mc{K}_0)(\mf{H}_{\infty}(V))=V_{0,\bullet}.
  \end{equation}
We first prove the admissibility of $\mf{H}_{\infty}(V)$. Any $w\in \mf{H}_{\infty}(V)$ can be written as $w=t^m v$ for some $m\in \mb{Z}$ and $v\in V_{l,n}$ for some integers $l,n$. When $a\leq m$, by \cref{def:plqi} we have $Q_{a,b}t^m v=t^m q_{a-m,b}v$. By \cref{def:ind} and the induction, $q_{a-m,b}v\in V_{l+1,n-a+m}$, which is $0$ when $a\ll 0$ due to the boundedness property of $V$. Hence $\mf{H}_{\infty}(V)$ is admissible. Next, we prove \cref{eq:73}. We notice that 
\begin{equation*}
    (\mc{H}_0\cap \mc{K}_0)(\mf{H}_{\infty}(V))=\{x\in V \mid q_{a,i}x=0 \text{ for all } a\leq 0, i\in [r]\} 
\end{equation*}
where $q_{a,i}$ is defined in \cref{def:plqi}.  It obviously contains $V_{0,\bullet}$. On the other hand, we notice that $(\mc{H}_0\cap \mc{K}_0)(\mf{H}_{\infty}(V))$ is closed under $f_{[r]}$ and by \cref{lem:45} and \cref{eq3}
\begin{equation*}
  e_{[r]}f_{[r]}|_{(\mc{H}_0\cap \mc{K}_0)(\mf{H}_{\infty}(V))}=(q_{[r]}p_{[r]})^2|_{(\mc{H}_0\cap \mc{K}_0)(\mf{H}_{\infty}(V))}=id.
\end{equation*}
Hence for any $v\in (\mc{H}_0\cap \mc{K}_0)(\mf{H}_{\infty}(V))$, we have
\begin{equation*}
  v=e_{[r]}f_{[r]}v=e_{[r]}^2f_{[r]}^2v=\cdots =e_{[r]}^m f_{[r]}^m v\cdots 
\end{equation*}
and thus $v\in V_{0,\bullet}$, as the action of $f_{[r]}$ is nilpotent on $V_{l,n}$ for any $l<0$.

 The isomorphisms \cref{eq:standard} follow from \cref{eq:73} and \cref{thm:equivalence}. To prove \cref{eq:74}, by \cref{eq:standard} we can assume that $V\cong \mf{F}(r)^{\leq 0}$. What we need to prove is that for any integers $l,n$, when $m\gg 0$,
  \begin{equation*}
    \mf{F}(r)^{\leq 0}_{l-rm,n-ml+\frac{1}{2}(m-1)mr}= \mf{F}(r)_{l-rm,n-ml+\frac{1}{2}(m-1)mr},
  \end{equation*}
Let $\mf{F}(r)^{\leq m}:=\mc{H}_{m}(\mf{F}(r))$. Then under the action of $E^m$, we have
\begin{align*}
  \mf{F}(r)_{l,n}\cong \mf{F}(r)_{l-rm,n-ml+\frac{1}{2}(m-1)mr},\quad 
  \mf{F}(r)^{\leq m}_{l,n} \cong \mf{F}(r)_{l-rm,n-ml+\frac{1}{2}(m-1)mr}^{\leq 0}.
\end{align*}
Hence we only need to prove that for any integers $l,n$, $\mf{F}^{\leq m}(r)_{l,n}=\mf{F}(r)_{l,n}$ when $m\gg 0$.
It follows from direct computation on the Maya diagrams. 
\end{proof}

\begin{remark}
  More information about the bounded bigraded $\mc{E}(r)$-representations can be found in \cref{sec:visualization}, where we provide a visualization for most of the morphisms introduced in this subsection and how the parabolic trajectories appear in the action of $e$. 
\end{remark}

\subsection{The verification of \texorpdfstring{\cref{stable:limit}}{}}
Now we verify \cref{stable:limit} and its cohomological variants. Due to the rigidity of \cref{prop:big}, we can prove much stronger statements.
\begin{proposition}
\label{prop:limit}
    We have an isomorphism of $\mc{E}(r)$-representations
    \begin{equation}
      \label{align1}
      \bigoplus_{l,n\in \mb{Z}}\mathrm{G}_0(\mathrm{M}^{0}(l,n))\cong \bigoplus_{n\in \mb{Z}}\mathrm{G}_0(\mc{M}_{H}(n))\otimes \mf{F}(r)^{\leq 0}.
    \end{equation}
    Moreover, under the functor $\mf{H}_{\infty}$, we have an isomorphism of $\mf{C}(r)$-representations
\begin{align}
  \label{align2}  \bigoplus_{l,n\in \mb{Z}}\mathrm{G}_0(\mc{M}_{H_{\infty}}(l,n))\cong \bigoplus_{n\in \mb{Z}}\mathrm{G}_0(\mc{M}_H(n))\otimes \mf{F}(r).
\end{align}
The isomorphisms of representations are bigraded, where the grading $l$ is compatible with the central charge grading, and the grading $n$ is compatible with the weight grading. A similar argument holds for Hochschild homology if we assume \cref{assumptionS}; in this case $\mf{F}(r)$ has homological degree $0$.
\end{proposition}
\begin{proof}
  It directly follows from \cref{shift1} and \cref{prop:big}.  We only need to remind that \cref{align2} follows from \cref{thm:NY2} and \cref{eq:74}. 
\end{proof}

To formulate the cohomological variants of \cref{prop:limit}, we need to define a cohomological grading on $\mf{F}(r)$ such that  the operators $p_{0,i},q_{0,i},e,f,e_{[r]},f_{[r]}$ on $\mf{F}(r)^{\leq 0}$ shift the cohomological grading by
  \begin{equation*}
  i,-i,-rl+\frac{1}{2}r(r-1),r(l+r)-\frac{1}{2}r(r-1),lr,-lr, 
\end{equation*}
which matches the degree shift in \cref{thm:4.4}.
The grading is given by 
\begin{equation}
\label{coho}
  \mathrm{coho}_{\mf{I}}:=\sum_{m\in \mb{Z}_{\leq 0}}(\frac{r(r-1)}{2}-\mathrm{wt}(I_m))-\sum_{m\in \mb{Z}_{>0}}\mathrm{wt}(I_m)+rw_{r,\mf{I}}.
\end{equation}
for any $r$-Maya diagram $\mf{I}$. The proof of the following argument is similar to \cref{prop:limit}.
\begin{proposition}
\label{prop:limit2}
     Let 
  \begin{align*}
    \mb{CH}^m(\mathrm{M}^0(l,n)):=\mathrm{CH}^{m-\frac{1}{2}(l(l+1))}(\mathrm{M}^0(l,n)), \\ \mb{CH}^m(\mc{M}_{H_{\infty}}(l,n)):=\mathrm{CH}^{m-\frac{1}{2}(l(l+1))}(\mc{M}_{H_{\infty}}(l,n)).
  \end{align*}
  Then there is an isomorphism of bounded triply graded $\mc{E}(r)$-representations on the Chow groups
  \begin{equation*}
    \bigoplus_{l,n\in \mb{Z}}\mb{CH}^*(\mathrm{M}^{0}(l,n))\cong \bigoplus_{n\in \mb{Z}}\mb{CH}^*(\mc{M}_{H}(n))\otimes \mf{F}(r)^{\leq 0}.
  \end{equation*}
  Here, the gradings on $l$ and $n$ are the same as \cref{prop:limit} and the cohomological grading on $\mf{F}(r)$ is induced by \cref{coho}. Under the $\mf{H}_{\infty}$ functor, we have an isomorphism of $\mf{C}(r)$-representations
  \begin{equation*}
    \bigoplus_{n\in \mb{Z}}\mb{CH}^*(\mc{M}_{H}(r,c_1,n))\otimes \mf{F}(r)\cong \bigoplus_{l,n\in \mb{Z}}\mb{CH}^*(\mc{M}_{H_{\infty}}(l,n)).
  \end{equation*}
  A similar argument holds for the Hodge cohomology if we replace the cohomology grading by bigrading.
\end{proposition}
\section{The proof of main theorems and the Boson--Fermion correspondence}
\label{sec:main}
Now we state our main theorems. We follow the notations in
\cref{sec:fermion} for the representation theory and \cref{sec:gra} for the moduli space of stable coherent sheaves and assume that $\mathrm{gcd}(r,c_1\cdot H)=1$.  Given $r>0$, 
  the \textbf{affine Lie algebra} $\hat{\gl}_r$ is defined by the generators $e_{ij}^{m}, m\in \mb{Z}, i,j\in [r]$, central element $K$ and a derivation $d$ such that 
\begin{gather}
  \label{eq:affine}
      [e_{ij}^a, e_{kl}^b]=(\delta_{jk}e_{il}^{a+b}-\delta_{il}e_{kj}^{a+b})+a\delta_{il}\delta_{jk}\delta_{a,-b}K,  \\
      \label{eq:derivation}
    [d,  e_{ij}^{l}]=l e_{ij}^{l}, \quad [d,K]=0.
  \end{gather}
\begin{theorem}[\cref{thm:1.1}]
      \label{main1}  We have an isomorphism of bounded bigraded $\hat{\mathrm{gl}}_r$-representations where the bi-grading on $\mf{F}(r)$ is given by the central charge and the weight grading:
\begin{equation}
  \label{eq:main11}
\bigoplus_{l,n\in \mb{Z}}\mathrm{G}_0(\mc{M}_{H_{\infty}}(l,n))\cong \bigoplus_{n\in \mb{Z}}\mathrm{G}_0(\mc{M}_H(n))\otimes \mf{F}(r).
\end{equation}
If \cref{assumptionS} is satisfied, then a similar argument also holds for the Hochschild homology, where $\mf{F}(r)$ always has homological degree $0$. 

Still assuming \cref{assumptionS}, we consider the degree shift 
\begin{equation*}
    \mb{CH}^m(\mc{M}_{H_{\infty}}(l,n)):=\mathrm{CH}^{m-\frac{1}{2}(l(l+1))}(\mc{M}_{H_{\infty}}(l,n)).
\end{equation*}
We have an isomorphism of triply graded $\hat{\mathrm{gl}}_r$-representations
  \begin{equation}
    \label{eq:main2} 
    \bigoplus_{n\in \mb{Z}}\mb{CH}^*(\mc{M}_{H}(r,c_1,n))\otimes \mf{F}(r)\cong \bigoplus_{l,n\in \mb{Z}}\mb{CH}^*(\mc{M}_{H_{\infty}}(l,n))
  \end{equation}
  where the bi-grading on $l,n$ is given by the central charge and the weight grading, and the cohomological grading is given by \cref{coho}. Similar arguments hold for Hodge cohomology as long as we replace the cohomological grading with a bi-grading.
\end{theorem}

We notice that \cref{eq:main11} and \cref{eq:main2} have been proved as isomorphisms of $\mf{C}(r)$-representations in \cref{prop:limit,prop:limit2}. The remaining step is to induce the $\hat{\mathrm{gl}}_r$-representation from the $\mf{C}(r)$ action. This is the famous Boson--Fermion correspondence of I. Frenkel \cite{FRENKEL1981259}, where the Fermionic Fock space $\mf{F}(r)$ can also be decomposed as the direct sum of \textbf{basic representations} of $\hat{\gl}_r$ with different central charges:
 \begin{theorem}[Boson--Fermion correspondence, Theorem \RN{1} 4.7 of \cite{FRENKEL1981259}]
  \label{thm:fermion-boson}
  Consider the following operators over $\mf{F}(r)$:
  \begin{equation*}
    e_{ij}^a:= \sum_{m\in \mb{Z}} :Q_{m,i}P_{m+a,j}:,  i,j\in [r], a\in \mb{Z},
  \end{equation*}
  where 
  \begin{equation*}
    :Q_{m,i}P_{n,j}: \,=\begin{cases}
      Q_{m,i}P_{n,j}, \text{ if } n>0, \\
      -P_{n,j}Q_{m,i}, \text{ if } n\leq 0.
    \end{cases}
  \end{equation*}
  With identity operator $K$ and a derivation operator $d$, 
  the operators $e_{ij}^a, d, K$ form a representation of the affine Lie algebra $\hat{\gl}_r$. In particular, the central charge degree is preserved in this representation. Moreover, the space 
  \begin{equation*}
    \mf{F}(r)_{l,\bullet}:=\bigoplus_{n\in \mb{Z}}\mf{F}(r)_{l,n}
  \end{equation*}
   is the basic representation of affine $\mathrm{gl}_r$ with central charge $l$.
\end{theorem}

\appendix
\section{Visualizations of bounded bigraded $\mathcal{E}(r)$-representations}
\label{sec:visualization}
This appendix presents visualizations of bounded bigraded $\mathcal{E}(r)$-representations
\[  
V=\bigoplus_{l\le 0,\; n\in\mathbb{Z}} V_{l,n},
\]
in \cref{sec:boundedrep} by placing the graded pieces on the affine $(n,-l)$-plane and tracing the parabolic trajectories generated by the $e$-action, with particular attention to the geometric case $V_{l,n}=\mb{H}^*(\mathrm{M}^0(l,n))$.

\subsection{The affine $(n,-l)$-plane}
Let $\mathrm{M}^0(l,n)$ denote the moduli of stable perverse coherent sheaves from \cref{sec:gra}. Then each $\mathrm{M}^0(l,n)$ fits into a diagram
\[
\begin{tikzcd}
& \mathrm{M}^0(l,n) \arrow[dl, "\eta"'] \arrow[dr, "\zeta"] & \\
\mathcal{M}_H(n) & & \mathcal{M}_H(n-l)
\end{tikzcd}
\]
where $\eta$ and $\zeta$ are Grassmannian of rank $(-l)$ associated to the complexes $\mathcal{U}_0$ and $\mathcal{U}_0^\vee[1]$ (see \cref{thm:NY1}). We visualize all the moduli spaces $\mathrm{M}^0(l,n)$ by placing each $\mathrm{M}^0(l,n)$ at the lattice point on the affine $(x,y) = (n,-l)$-plane as in Figure~\ref{figure:affine-plane-e-action}; lattice points on the $n$-axis precisely correspond to $\mathrm{M}^0(0,n)=\mathcal{M}_H(n)$.

With this coordinate choice (so the $-l$-direction appears at $45^\circ$), the diagonal arrows record the geometric maps: the left-down $45^\circ$ arrow from $(n,l)$ to $(n,0)$ represents $\eta$, and the right-down $45^\circ$ arrow from $(n,l)$ to $(n-l,0)$ represents $\zeta$.
\begin{figure}[ht]
  \centering
  \begin{tikzpicture}[line cap=round, line join=round]

    \clip (-0.5,-0.5) rectangle (9.75cm, 3.75cm);

    \pgftransformcm{1}{0}{0.5}{0.5}{\pgfpoint{0cm}{0cm}}

    \draw[style=help lines, dashed] (-8,-2) grid[step=1cm] (26,15);
    \draw[->, gray] (-2,0) -- (9.6,0) node[below=0.2cm, left=0.1cm] {$x=n$};
    \draw[->, gray] (0,-2) -- (0,6.7) node[above=0.1cm,  left=0.1cm] {$y=-l$};

    \foreach \k in {-9,-8,...,27} {
      \pgfmathsetmacro{\lmin}{max(-1, \k-25)}
      \pgfmathsetmacro{\lmax}{min(15, \k+15)}
      \pgfmathsetmacro{\nA}{\k - \lmin}
      \pgfmathsetmacro{\lA}{\lmin}
      \pgfmathsetmacro{\nB}{\k - \lmax}
      \pgfmathsetmacro{\lB}{\lmax}
      \draw[style=help lines, dashed] (\nA,\lA) -- (\nB,\lB);
    }


    \foreach \nzero in {0}{
      \draw[red, thick, dashed]
        plot[domain=6:12, variable=\lv]
          ({(\lv*(\lv-2))/4 + \nzero}, {\lv});
    }


    \tikzset{
      ptred/.style={
        draw=black,
        circle,
        inner sep=1.2pt,
        fill=black
      },
      ptblue/.style={
        draw=blue,
        circle,
        inner sep=1.2pt,
        fill=blue
      },
      edgearrow/.style={
        -{Latex[length=2.6mm,width=1.8mm]},
      }
    }


    \draw[red, thick, dashed, edgearrow]
      plot[smooth, domain=0:1, samples=120, variable=\t]
        ({\t*(\t-1)}, {2*\t});
    \node[red!70!black, above left] at (-0.1,1.0) {$e$};

    \draw[red, thick, dashed, edgearrow]
      plot[smooth, domain=1:2, samples=120, variable=\t]
        ({\t*(\t-1)}, {2*\t});
    \node[red!70!black, above] at (1,3.1) {$e$};

    \draw[red, thick, dashed, edgearrow]
      plot[smooth, domain=2:3, samples=120, variable=\t]
        ({\t*(\t-1)}, {2*\t});
    \node[red!70!black, above right] at (3.7,5) {$e$};





    \draw[->, thin] (0,2) -- (2,0) node[midway, above right] {$\zeta$};
    \draw[->, thin] (2,4) -- (6,0) node[midway, above right] {$\zeta$};

    \draw[->, thin] (2,4) -- (2,0) node[midway, below right] {$\eta$};
    \draw[->, thin] (6,6) -- (6,0) node[midway, below right] {$\eta$};

    \node[ptred] at (0,0) {};
    \node[ptred] at (0,2) {};
    \node[ptred] at (2,4) {};
    \node[ptred] at (6,6) {};
    
     \node[ptred] at (2,0) {};
	\node[ptred] at (6,0) {};


    \node[above] at (2,4) {$\mathrm{M}^0(l,n)$};
    \node[below] at (2,0) {$\mc{M}_H(n)$};
    \node[below] at (6,0) {$\mc{M}_H(n-l)$};
    \node[above] at (5.5,6) {$\mathrm{M}^0(l-r, n-l)$};

  \end{tikzpicture}
  \caption{The affine $(n,-l)$-plane. The red dashed curve indicates the parabolic trajectory passing through $\mathrm{M}^0(l,n)$.}
  \label{figure:affine-plane-e-action}
\end{figure}
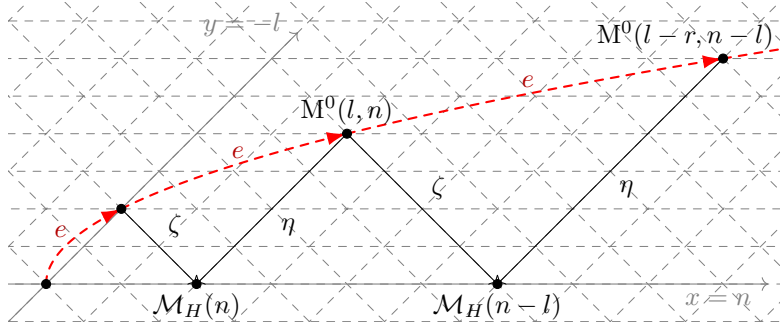

\subsection{Parabolic trajectories of the $e$-action}

More generally, let $$V=\bigoplus_{l \le 0,\; n\in\mathbb{Z}} V_{l,n}$$ be a bigraded $\mathcal{E}(r)$-representation as in \cref{sec:fermion}. We can similarly visualize $V$ by placing each $\mathbb{Z}$-module $V_{l,n}$ at the lattice point on the affine $(n,-l)$-plane, as in Figure~\ref{figure:affine-plane-e-action}. Iterating the action of $e\in\mathcal{E}(r)$ then produces
\[
V_{l,n}\xrightarrow{e}V_{l-r,\,n-l}\xrightarrow{e}V_{l-2r,\,n-2l+r}\xrightarrow{e}\cdots
\xrightarrow{e} V^{(m)}_{l,n}\xrightarrow{e} V^{(m+1)}_{l,n}\xrightarrow{e}\cdots,
\]
where
\[
V^{(m)}_{l,n}:=V_{\,l-mr,\; n-ml+\tfrac{m(m-1)}{2}r}
\]
corresponds to $t^m V(l,n)$ from \cref{def:bounded.bigraded}. Thus the orbit
\[
\mathbb{O}(l,n):=\{\,V^{(m)}_{l,n}\mid m \in \mathbb{Z}_{\ge 0} \,\}
\]
is represented by the sequence of lattice points lying on the parabola
\[
(x-n)=\frac{1}{2r}\,(y+l)\,(y-r-l)
\]
in the $(x,y)=(n,-l)$-plane; see Figure \ref{figure:affine-plane-e-action}. We call such a sequence $\mathbb{O}(l,n)$ a \textbf{parabolic trajectory}. Note that the parabolic trajectories $\mathbb{O}(l,n)$ and $\mathbb{O}(-r-l,n)$ lie on the same parabola, but they are disjoint unless $l\equiv -r-l\pmod r$.

Every $\mb{Z}$-module $V_{l',n'}$ lies on a unique trajectory $\mathbb{O}(l, n)$ with $1-r\le l \le 0$ and $n\in\mathbb{Z}$, and distinct trajectories $\mathbb{O}(l, n)$ with $1\!-\!r\le l\le 0$ are disjoint. Hence the lattice  labeled by $V_{l,n}$ is \textbf{foliated} by the parabolic trajectories $\mathbb{O}(l,n)$, indexed by $n\in\mathbb{Z}$ and $1-r\le l\le 0$.

In the geometric situation where $V_{l,n}=\mb{H}^* (\mathrm{M}^0(l,n))$, the Bogomolov--Gieseker inequality yields boundedness: there exists $n_0$ such that $V_{l,n}=0$ for all $n<n_0$. Moreover, the Chern class of the sheaves parametrized by $\mathrm{M}^0(l,n)$ satisfy the Bogomolov--Gieseker  inequality if and only if the Chern class of the sheaves parametrized by
\[
\mathrm{M}^m(l,n)\cong \mathrm{M}^0\!\big(l-mr,\; n-ml+\tfrac{m(m-1)}{2}\,r\big)
\]
does. Consequently, the potentially nonempty moduli $\mathrm{M}^0(l,n)$ occur precisely at lattice points lying on the parabolic trajectories $\mathbb{O}(l,n)$ with $n\ge n_0$ and $1-r\le l\le 0$; see \cref{figure:parabolas:r=2} for an illustration when $r=2$.


\begin{figure}[ht]
  \centering
  \begin{tikzpicture}[line cap=round, line join=round]

    \clip (-0.5,-0.5) rectangle (12.75cm, 3.75cm);

    \pgftransformcm{1}{0}{0.5}{0.5}{\pgfpoint{0cm}{0cm}}

    \draw[style=help lines, dashed] (-8,-2) grid[step=1cm] (26,15);
    \draw[->, gray] (-2,0) -- (12.6,0) node[below=0.1cm] {$n$};
    \draw[->, gray] (0,-2) -- (0,6.7) node[above left] {$-l$};
      
          \node[below=0.1cm, gray] at (0,0) {$n_0$};
          \node[below=0.05cm, gray] at (1,0) {$n_0+1$};
                    \node[below=0.1cm, gray] at (2,0) {$\cdots$};

 \foreach \k in {-9,-8,...,26, 27} {
      \pgfmathsetmacro{\lmin}{max(-1, \k-25)}
      \pgfmathsetmacro{\lmax}{min(15, \k+15)}
     \pgfmathsetmacro{\nA}{\k - \lmin}
       \pgfmathsetmacro{\lA}{\lmin}
        \pgfmathsetmacro{\nB}{\k - \lmax}
      \pgfmathsetmacro{\lB}{\lmax}
        \draw[style=help lines, dashed] (\nA,\lA) -- (\nB,\lB);}


\foreach \nzero in {0,...,12}{
  \draw[red, thin, domain=0:{5/3}, samples=50, smooth, variable=\t]
        plot ({\t*(\t-1) + \nzero}, {2*\t});
      \draw[red, thin]
        plot[domain=3:15, variable=\lv]
          ({(\lv*(\lv-2))/4 + \nzero}, {\lv});
    }

\foreach \nzero in {0,...,12}{
      \draw[blue, thin, domain=0:2, samples=60, smooth, variable=\t]
        plot ({\t*(\t-1) + \t + \nzero}, {2*\t + 1});
      \draw[blue, thin]
        plot[domain=5:15, variable=\lv]
          ({((\lv-1)^2)/4 + \nzero}, {\lv});
    }

\tikzset{
  ptred/.style={
    draw=red, 
    circle, 
    inner sep=0.8pt, 
    fill=red, 
  },
  ptblue/.style={
    draw=blue, 
    circle, 
    inner sep=0.8pt, 
    fill=blue, 
  },
  ptorange/.style={
    draw=orange, 
    circle, 
    inner sep=1.2pt, 
    fill=orange, 
    fill opacity=0.6, 
    draw opacity=0.8
  },
  }

   \foreach \l in {-1,...,12} {
      \pgfmathtruncatemacro{\lmod}{mod(\l,2)}
      \foreach \n in {-1,...,12} {

        \ifnum\lmod=0\relax
          \pgfmathsetmacro{\nthresh}{(\l*(\l-2))/4}
          \ifdim \n pt<\nthresh pt\relax\else
            \node[ptred] at (\n,\l) {};
          \fi
        \fi

        \ifnum\lmod=1\relax
          \pgfmathsetmacro{\nthresh}{((\l-1)^2)/4}
          \ifdim \n pt<\nthresh pt\relax\else
            \node[ptblue] at (\n,\l) {};
          \fi
        \fi

      }

    }

  \end{tikzpicture}
\caption{Visualization for a bounded bigraded $\mathcal{E}(r)$-representation $V$ in the case $r=2$. Each $\mb{Z}$-module $V_{l,n}$ is placed at the lattice point $(n,-l)$. Points with $l\equiv 0 \pmod{2}$ are shown in red, and those with $l\equiv 1 \pmod{2}$ in blue. The region of potentially nonzero $V_{l,n}$ is foliated by parabolic trajectories $\mathbb{O}(l,n)$ with $l\in\{-1,0\}$ (and $n\ge n_0$ in the geometric situation $V_{l,n}=\mb{H}^*(\mathrm{M}^0(l,n))$). Parabolic trajectories $\mathbb{O}(l,n)$ with $l=0$ are drawn in red, and those with $l=-1$ in blue.}
 \label{figure:parabolas:r=2}
\end{figure}
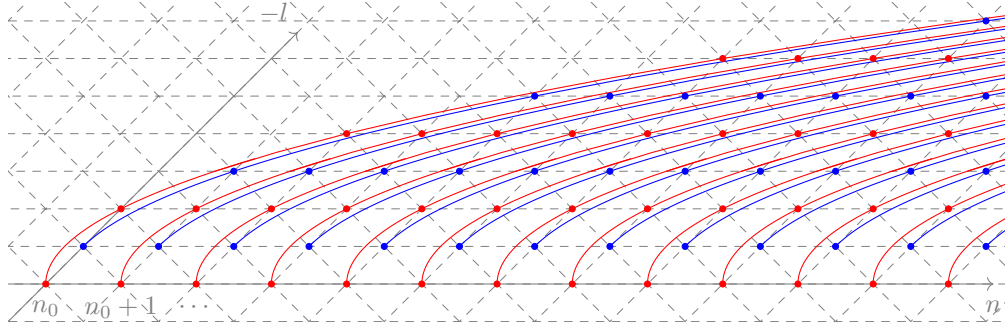

\subsection{The $\mathcal{E}(r)$-action}
The actions of $e$ (and hence $f$) are shown in Figures~\ref{figure:affine-plane-e-action} and \ref{figure:parabolas:r=2}. In particular, $e$ moves forward along each parabolic trajectory $\mathbb{O}(l,n)$, while $f$ moves backward. In these figures, the family of operators $q_{a,i}$ with $a\le 0$ from \cref{def:ind} can be depicted as follows, exhibiting a pattern similar to differentials on successive pages of a spectral sequence:
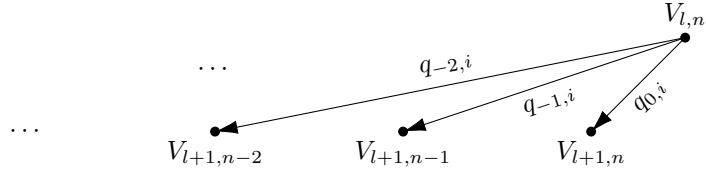
\[
\begin{tikzpicture}[scale=1.25]
    \tikzset{
      edgearrow/.style={
        -{Latex[length=3mm,width=1.8mm]} 
      }
    }
    
    \node[above] at (0,0) {$V_{l,n}$};
  \node[below] at (-1,-1) {$V_{l+1,n}$};
  \node[below] at (-3,-1) {$V_{l+1,n-1}$};
  \node[below]  at (-5,-1){$V_{l+1,n-2}$};
  \node at (-7,-1)  {$\cdots$};
    \node[above] at (-5,-0.5)  {$\cdots$};

  \fill (0,0) circle (1.5pt);
  \fill (-1,-1) circle (1.5pt);
  \fill (-3,-1) circle (1.5pt);
  \fill (-5,-1) circle (1.5pt);

  \draw[edgearrow] (0,0) -- (-1,-1) node[midway, below, sloped] {$q_{0,i}$};
  \draw[edgearrow] (0,0) -- (-3,-1) node[midway, below, sloped] {$q_{-1,i}$};
  \draw[edgearrow] (0,0) -- (-5,-1) node[midway, above, sloped] {$q_{-2,i}$};
\end{tikzpicture}
\]
The operators $p_{a,i}$ act in the directions opposite to the corresponding $q_{a,i}$.

\subsection{Stable colimits and $\mathfrak{E}(r)$-action}
By \cref{prop:big}, if $V$ is a bounded bigraded $\mathcal{E}(r)$-representation, then each parabolic trajectory $\mathbb{O}(l,n)$ stabilizes after finitely many steps:
\[
\mathfrak{H}_{\infty}(V)_{l,n} \cong V_{l,n}^{(m)} \quad \text{for all } m \gg 0.
\]
In the geometric setting $V_{l,n}=\mb{H}^*(\mathrm{M}^0(l,n))$, \cref{prop:limit} identifies the stable colimit as
\(
\mathfrak{H}_{\infty}(V)_{l,n} = \mb{H}^*(\mc{M}_{H_\infty}(l,n)).
\)

In the colimit $\mathfrak{H}_{\infty}(V)$, the operator $e$ is (formally) invertible, so each stabilized parabolic orbit as in Figure \ref{figure:parabolas:r=2} extends to the negative direction $\{\mathfrak{H}_{\infty}(V)_{l,n}^{(m)}\}_{m\in\mathbb{Z}}$, and all transition maps along the orbit are isomorphisms. The operators $E$ and $F$ move forward and backward, respectively, along each stabilized parabolic orbit. Moreover, the operators $p_{a,i}$ and $q_{a,i}$ for $m \le 0$ descend to well-defined operators $P_{a,i}$ and $Q_{a,i}$ on the colimit. Similar to $q_{a,i}$, the operators $Q_{a,i}$ for $a \ge 0$ can be visualized as follows:
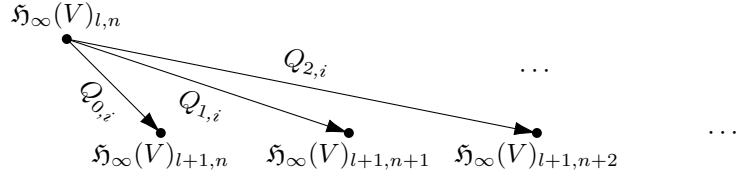
\[
\begin{tikzpicture}[scale=1.25]
  \tikzset{
    edgearrow/.style={-{Latex[length=3mm,width=1.8mm]}}
  }

  \node[above] at (0,0) {$\mathfrak{H}_\infty(V)_{l,n}$};
  \node[below] at (1,-1) {$\mathfrak{H}_\infty(V)_{l+1,n}$};
  \node[below] at (3,-1) {$\mathfrak{H}_\infty(V)_{l+1,n+1}$};
  \node[below] at (5,-1) {$\mathfrak{H}_\infty(V)_{l+1,n+2}$};
  \node at (7,-1) {$\cdots$};
  \node[above] at (5,-0.5) {$\cdots$};

  \fill (0,0) circle (1.5pt);
  \fill (1,-1) circle (1.5pt);
  \fill (3,-1) circle (1.5pt);
  \fill (5,-1) circle (1.5pt);

  \draw[edgearrow] (0,0) -- (1,-1) node[midway, below, sloped] {$Q_{0,i}$};
  \draw[edgearrow] (0,0) -- (3,-1) node[midway, below, sloped] {$Q_{1,i}$};
  \draw[edgearrow] (0,0) -- (5,-1) node[midway, above, sloped] {$Q_{2,i}$};
\end{tikzpicture}
\]
The operators $P_{a,i}$ act in the directions opposite to the corresponding $Q_{a,i}$.

\end{document}